\documentclass[11pt, a4paper]{article}

\usepackage{inputenc} 
\usepackage[english]{babel}
\usepackage{dsfont}
\usepackage{setspace}
\usepackage{tocloft}
\usepackage[totalheight=24 true cm, totalwidth=17 true cm]{geometry}
\usepackage{enumitem}

\usepackage{bbding}
\usepackage{amsmath}
\usepackage{amsthm}
\usepackage{amssymb}
\usepackage{mathrsfs}
\usepackage{mathtools}
\usepackage{ stmaryrd }
\usepackage{url}
\usepackage{wrapfig}
\usepackage{starfont}
\usepackage{pifont}
\usepackage{eurosym}
\usepackage{subfloat}

\usepackage{siunitx}
\usepackage{xcolor}
\usepackage{subfig}

\usepackage[maxbibnames=99]{biblatex}
\usepackage{imakeidx}
\makeindex[]

\usepackage{multicol}
\usepackage{tikz,pgf}
\usepackage{tikz-cd}
\usepackage{wasysym}
\usepackage{graphicx}
\usepackage{fancyhdr}
\usepackage{verbatim}

\usepackage[all,dvips,knot,web,arc,curve,color,frame]{xy}
\usepackage[all,knot,arc,color,web]{xy}
\xyoption{arc}
\xyoption{web}
\xyoption{curve}

\numberwithin{equation}{section}

\theoremstyle{plain}
\newtheorem{theorem}{Theorem}[section]
\newtheorem{proposition}[theorem]{Proposition}
\newtheorem{corollary}[theorem]{Corollary}
\newtheorem{lemma}[theorem]{Lemma}
\newtheorem*{claim}{Claim}
\newtheorem*{theorem*}{Theorem}

\theoremstyle{definition}
\newtheorem{definition}[theorem]{Definition}
\newtheorem{construction}[theorem]{Construction}

\newtheorem{remark}[theorem]{Remark}

\newtheorem{notation}[theorem]{Notation}
\newtheorem{example}[theorem]{Example}
\newtheorem{question}[theorem]{Question}
\usepackage[bookmarksnumbered=true]{hyperref} 
\hypersetup{
     colorlinks = true,
     linkcolor = blue,
     anchorcolor = blue,
     citecolor = teal,
     filecolor = blue,
     urlcolor = blue
     }

\DeclareMathOperator{\cl}{cl}
\DeclareMathOperator{\comb}{comb}
\DeclareMathOperator{\diag}{diag}

\DeclareMathOperator{\rk}{rk}

\DeclareMathOperator{\Span}{Span}
\DeclareMathOperator{\sgn}{sgn}

\DeclareMathOperator{\B}{\mathscr{B}}
\DeclareMathOperator{\C}{\mathscr{C}}
\DeclareMathOperator{\CC}{\mathbb{C}}
\DeclareMathOperator{\CCC}{\mathcal{C}}
\DeclareMathOperator{\D}{\mathscr{D}}
\newcommand{\F}{\mathscr F}

\newcommand{\I}{\mathscr I}
\DeclareMathOperator{\II}{\mathcal{I}}
\DeclareMathOperator{\kk}{\mathbf{k}}
\DeclareMathOperator{\LL}{\mathcal{L}}
\DeclareMathOperator{\LLL}{\mathscr L}
\DeclareMathOperator{\PP}{\mathbb{P}}
\DeclareMathOperator{\PPPP}{\mathcal{P}}

\DeclareMathOperator{\ZZ}{\mathbb{Z}}

\usepackage[maxbibnames=99]{biblatex}
\title{Liftable Point-Line Configurations: Defining Equations and Irreducibility of Associated Matroid and Circuit Varieties}
\author{Oliver Clarke, Giacomo Masiero and Fatemeh Mohammadi
}

\begin{document}
\maketitle

\begin{abstract}
We study point-line configurations through the lens of projective geometry and matroid theory. Our focus is on their realisation spaces, where we introduce the concepts of liftable and quasi-liftable configurations, exploring cases in which an $n$-tuple of collinear points can be lifted to a non-degenerate realisation of a point-line configuration. We show that forest configurations are liftable and characterise the realisation space of liftable configurations as the solution set of certain linear systems of equations. Moreover, we study the Zariski closure of the realisation spaces of liftable and quasi-liftable configurations, known as matroid varieties, and establish their irreducibility. Additionally, we compute an irreducible decomposition for their corresponding circuit varieties. Applying these liftability properties, we present a procedure generate some of the defining equations of the associated matroid varieties. As corollaries, we provide a geometric representation for the defining equations of two specific examples: the quadrilateral set and the $3\times4$ grid. While the polynomials for the latter were previously computed using specialised algorithms tailored for this configuration, the geometric interpretation of these generators was missing. We compute a minimal generating set for the corresponding ideals.
\end{abstract}

{\hypersetup{linkcolor=black}
\setcounter{tocdepth}{1}
\setlength\cftbeforesecskip{1.1pt}
{\tableofcontents}}

\section{Introduction}

The axiomatic definition of matroids was established in 1935 by Whitney \cite{Whi}, with MacLane highlighting their intimate connection with projective geometry soon after \cite{Mac}.
One prominent tool in this context is the Grassmann-Cayley algebra, which constructs polynomial equations from a given set of synthetic projective geometric statements, see e.g.~\cite{white1995tutorial, white2017geometric}. Here, we provide other geometric tools for constructing such polynomials and apply them in specific examples to demonstrate that the constructed polynomials minimally generate the corresponding ideal.

\medskip
A matroid, denoted by $M$, is a combinatorial object that extends the notion of linear independence from vector spaces. The matroid records all possible combinations of linearly independent vectors within a given set of vectors in a vector space, see \cite{Whi, Oxl}. If this process is reversible, meaning that given a matroid $M$, we can identify such a vector collection, these vectors are called a realisation of $M$. The space of all realisations of $M$ is denoted as $\Gamma_M$. The matroid variety $V_M$ of $M$ is defined as the Zariski closure of this realisation space.
This notion, introduced in \cite{GGMS}, gives rise to a deep combinatorial structure called the matroid stratification of the Grassmannian. However, understanding geometric properties of these strata, such as their irreducibility and defining equations \cite{STW}, is a challenging problem. So, it is natural to consider specific classes of matroids. For instance, the matroid varieties of
uniform matroids have been extensively studied in commutative algebra, in the context of determinantal varieties, see e.g.~\cite{bruns1990number, bruns2003determinantal, sturmfels1990grobner, Fatemeh, Fatemeh2}.

\smallskip

In this work, we focus on matroids of rank three, represented by point-line configurations, and use tools from incidence and projective geometry to study their associated varieties and defining polynomial equations. Specifically, we develop methods to translate the incidence structure of the underlying configuration into a geometric representation for their polynomials.

\smallskip

The matroid varieties arising from point-line configurations are a rich and diverse family. For example, the Mn\"ev-Sturmfels Universality Theorem \cite{LeVa, Mne, Stu1} shows that matroid varieties satisfy Murphy's law in algebraic geometry, i.e., given any singularity of a semi-algebraic set, there is a matroid variety with the same singularity, up to a mild equivalence on singularities. See also \cite{bokowski1989computational}. Here, the matroid varieties associated with a point-line configuration can achieve all such singularities. Additionally, the extra structure induced by point-line configurations is mirrored in other contexts such as conditional independence constraints in algebraic statistics \cite{DSS, CGMM, CMM, caines2022lattice,clarke2020conditional}.

\smallskip

While the relationship between matroids and projective geometry is now well-established \cite{Mac}, the utilisation of projective incidence geometry in investigating matroid varieties is a relatively recent development. For example, the Grassmann-Cayley algebra can be employed to construct some polynomial equations in the matroid ideal $I_M$; see the example below from~\cite{STW, Stu4}.
\begin{example}
    The associated ideal of the matroid in Figure~\ref{fig:3 conc} contains three degree-3 polynomials reflecting collinearities, alongside a degree-6 polynomial derived via the Grassmann-Cayley method.
\end{example}

Although, given a matroid $M$, classical tools such as the Grassmann-Cayley algebra can be employed to construct some polynomial equations in $I_M$ which are not determined by the circuits of the matroid, i.e.~lie in $I_M\backslash I_{\C(M)}$, the description of all such polynomials remains incomplete. This is because the construction of the ideal involves a saturation step that encodes the independence relations of the matroid, potentially introducing additional polynomials not necessarily derived from this method \cite{CGMM, STW}. 
The current algorithms to compute saturation of ideals has high complexity and, consequently, provide results only for small matroids. In addition, when results are obtained in this way, the outputs often consist of 
long polynomials that cannot be parsed by a human, which gives little geometric intuition, see the following two examples, the first one computed by Macaulay2~\cite{GrSt}.

\begin{example}\label{exa:quadset}
The ideal of the quadrilateral set in Figure~\ref{fig:introd} is generated by 14 polynomials:
    \begin{itemize}
        \item[a.] 4 of them are degree three, deduced directly from the collinearities in the point-line configuration.
        \item[b.] The remaining generators are degree-6  polynomials in 18 variables, consisting of 14 or 16 terms.
    \end{itemize}   
\end{example}

\begin{example}\label{exa:3x4grid}
Consider the matroid of the $3 \times 4$ grid from Figure~\ref{fig:intro_3x4g}.  To compute the associated ideal of this matroid, the standard Gr\"obner basis computation algorithms do not terminate. Pfister and Steenpass, in \cite{PfSt}, developed and optimised a specific algorithm for this case. Through numerical analysis, they demonstrated that the corresponding ideal has 44 generators:
    \begin{itemize}
        \item[c.] 16 of them are degree-3, deduced directly from collinearities in the point-line configuration.
        \item[d.] The remaining generators are degree-12 polynomials in 36 variables, consisting of $\approx 250$ terms.
    \end{itemize}
However, there is no geometric description of these polynomials in terms of Grassmann-Cayley algebra. \end{example}

\begin{figure}[h]
\subfloat[3 concurrent lines]{\label{fig:3 conc}
\centering
\includegraphics[scale = 0.3]{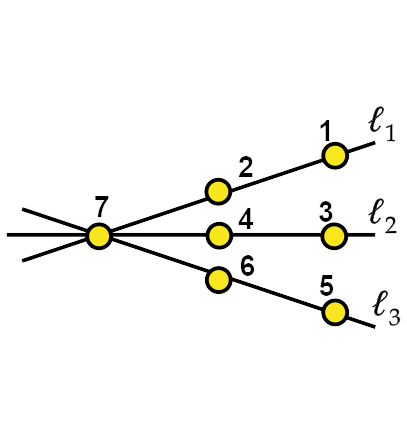}
}
\hfill
\subfloat[Quadrilateral set]{\label{fig:introd}
\centering
\includegraphics[scale = 0.3]{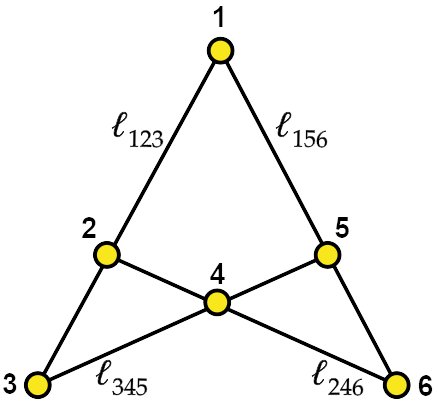}
}
\hfill
\subfloat[$3 \times 4$ Grid]{\label{fig:intro_3x4g}
\centering
\includegraphics[scale = 0.3]{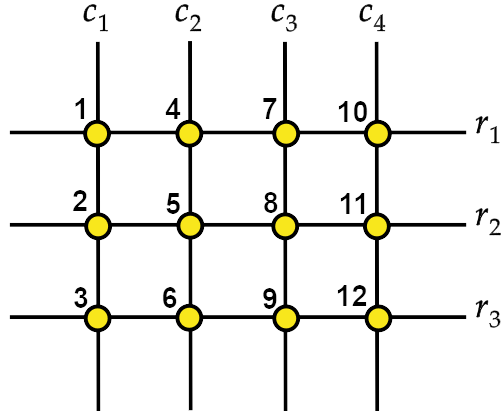}
} \hspace{30pt}
\caption{(a) 3 concurrent lines; (b) Quadrilateral set $\LL_{QS}$; (c) $3\times4$ grid $\LL_{G^3_4}$}
\label{fig:exa}
\end{figure}\medskip

\noindent Motivated by the above example, we investigate a new process to construct polynomials in the matroid ideal. 
Adopting an incidence-geometry viewpoint, we explore the conditions under which a tuple of collinear points in $\mathbb{P}^2_{\mathbb{C}}$ can be lifted to a non-trivial realisation of a given point-line configuration $\mathcal{C}$.
\begin{question}\label{que:liftability}
    Let $\CCC$ be a point-line configuration with $n$ points, $\ell$ a line in $\PP^2_{\CC}$ and $P$ a point in $\PP^2_{\CC} \setminus \ell$. Under what conditions is an $n$-tuple of points on $\ell$ the image, under the projection from $P$ to $\ell$, of a non-trivial realisation of $\mathcal{C}$?
\end{question}
When this is possible, we refer to $\mathcal{C}$ as \textit{liftable}; and when this is possible, up to removing a line, then we call $\CCC$ \textit{quasi-liftable}. See Definitions~\ref{def: liftable} and~\ref{def: quasi-liftable}. In Section~\ref{sec:circ}, we analyse such configurations and their associated matroid and circuits varieties and ideals.

\medskip
The following theorem summarises our main results from Section~\ref{sec:circ}. Below, we denote the ideal generated by the circuits of the matroid as $I_{\C(M)}$ and the matroid ideal as $I_M$. Similarly, we use the notation $V_{\C(M)}$ and $V_M$ for their associated varieties.

\begin{theorem}\label{thm:intro} Let $M$ be a rank-$3$ matroid whose associated point-line configuration $\CCC_M$ has no triplets of concurrent lines. Then,
\begin{itemize}
    \item The matroid variety $V_{M}$ is irreducible. \hfill{\rm(Theorem~\ref{decomposition liftable})}
    \item If $\CCC_M$ is liftable, then $V_{\C(M)} = V_M$ and $\sqrt{I_{\C(M)}} = I_{M}.$\hfill{\rm(Theorem~\ref{circuit variety of liftable conf.})}
    \item If $\CCC_M$ is connected quasi-liftable, then
        $V_{\C(M)} = V_0 \cup V_{M}$ and $\sqrt{I_{\C(M)}} = I_0 \cap I_{M}$ \hfill{\rm(Theorem~\ref{decomposition quasi-liftable})}
    
    where $V_0$ is the matroid variety whose associated configuration is a line with $n$ marked points. Furthermore, the decompositions are, respectively, irreducible and prime.
\end{itemize}
\end{theorem}
Specifically, Theorem~\ref{thm:intro} offers a geometric approach to generating certain polynomials in the ideal $I_M$, as demonstrated in Proposition~\ref{radical}. In Sections~\ref{sec:qua} and~\ref{sec:gri}, we illustrate how these polynomials suffice to generate the matroid ideal for the quadrilateral set and the $3 \times 4$ grid. Moreover, we address Question~\ref{que:liftability} for these two matroids. More precisely, we provide a characterisation of 6-tuples of collinear points that can be lifted to a quadrilateral set, and of 12-tuples that can be lifted to a $3 \times 4$ grid. For the quadrilateral set, this provides a new characterisation, equivalent to the one in \cite{Ric, Stu}.

Additionally, in the context of utilising incidence geometry to explore realisation spaces of matroids, we demonstrate that a subset of the introduced polynomials generates the ideal of the matroid varieties for both configurations. Notably, these polynomials rewrite the numerically obtained high-degree polynomials as sums of determinants involving the coordinates of the points.

\begin{theorem*}[Theorems~\ref{Generators} and~\ref{generators grid}]
    Let $\{R_1,R_2,R_3,U\}$ be a frame of reference in $\PP^2_{\CC}$, and let $\ell_{123}$ and $c_1$ be as in Figures~\ref{fig:introd}
and~\ref{fig:intro_3x4g}.
    \begin{itemize}
        \item The 10 generators in Example~\ref{exa:quadset}.b can be replaced by the following polynomials defined in \eqref{char1}:
        $$QS(\ell_{123}; R_i, R_j, R_k) \qquad \hbox{for any } i \leq j\leq k, \hbox{ with } i,j,k \in \{1,2,3\}.$$
        \item The 28 generators in Example~\ref{exa:3x4grid}.d can be replaced by the following polynomials from~\eqref{generator G34}: $$G^3_4(c_{1}; R_i, R_j, R_k, R_l, R_m,R_n) \qquad \hbox{for any }  i \leq j\leq k \leq l \leq m \leq n, \hbox{with } i,j,k,l,m,n \in \{1,2,3\}.$$
    \end{itemize}
\end{theorem*}

To the best of our knowledge, there is currently no established method for computing the equations defining the matroid varieties, aside from the Grassmann-Cayley method~\cite{white2017geometric,white1995tutorial,sitharam2017handbook}. Furthermore, current computer algebra programs cannot handle the computations. The above theorems provide a geometric representation of the generators. Furthermore, we prove that the generated polynomial form a minimal generating set for the corresponding matroid ideals.

\medskip

We conclude the introduction by giving an outline of the paper. In Section~\ref{sec:pre}, we fix our notations for matroids, matroid varieties, circuit varieties, and point-line configurations. 
In addition, it describes an explicit way of associating a matroid variety and a circuit variety with a point-line configuration. In Section~\ref{sec:circ}, we present our main results, providing an irreducible decomposition of the circuit variety associated with point-line configurations having certain liftability properties. In Sections~\ref{sec:qua} and~\ref{sec:gri}, we apply and complete the results of Section \ref{sec:circ} for the quadrilateral set and the $3 \times 4$ grid.

\section{Preliminaries}\label{sec:pre}

In this section, we provide background on the theory of matroids arising from point-line configurations and fix our notation. We recall some known results about matroid varieties and their defining equations. For further details, we refer the reader to \cite{Oxl, Stu4, CGMM}. In addition, as far as notions of commutative algebra are concerned, we always refer to the definitions in \cite{Har}.

\begin{notation} \label{general notation}
    Throughout the paper, for $n, d \in \ZZ$, $1 \le d \leq n$ and $\D \subseteq 2^{[n]}$, we write: 
    \begin{itemize}
        \item $[n]$ for the set $\{1,\dots,n\}$ and $\binom{[n]}{d}$ for the collection of $d$-subsets of $[n]$;
        
        \item $\min(\D) = \{D \in \D \mid \text{if } D' \in \D \text{ and } D'  \subseteq D \text{ then } D' = D\}$.
    \end{itemize}

    Let $\kk$ be a field and let $X = (x_{i,j})^{i = 1, \dots, d}_{j = 1, \dots, n}$ be a $d \times n$ matrix of variables. We denote by $R = \kk[X]$
    the polynomial ring in the variables $x_{i,j}$. In addition, if $A \subseteq [d]$, $B \subseteq [n]$, and $\# A = \# B$, then we denote by $[A|B]_X \in R$ the minor of the submatrix of $X$ with rows indexed by $A$ and columns indexed by $B$. If $\# A = \# B = d$, we write $[B]_X \in R$. \hfill $\diamond$
\end{notation}

\subsection{Matroids and matroid varieties}
\begin{definition}[Matroid]
    A \textit{matroid $M$} is the datum of a finite ground set $E$, which will typically be $[n]$, together with a non-empty collection $\B(M) \subseteq 2^E$ of bases, satisfying the basis exchange axiom: 
    \begin{center}
    \textrm{if $B,B' \in \B(M)$ and $\beta \in B \setminus B'$, then there exists $\beta'\in B' \setminus B$ such that $(B \setminus \{\beta\}) \cup \{\beta'\} \in \B(M)$.}
    \end{center}
\end{definition}

\begin{remark}
    There are other equivalent definitions of a matroid in terms of other axioms.
    The bases of a matroid can determine these other pieces of data and we list the ones of relevance to us below:
    \begin{itemize}
        \item $\I(M) := \{ I \subseteq E \mid I \subseteq  B \text{ for some } B \in \B(M) \}$ the \textit{independent sets} are subsets of the bases;
        
        \item $\D(M) := \{D \subseteq E \mid D \notin \I(M) \}$ the \textit{dependent sets} of $M$ are the non-independent sets;
        
        \item $\C(M) := \min(\D(M))$ the \textit{circuits} of $M$ are the minimal dependent sets;
        
        \item $\rk_M : 2^E \rightarrow \ZZ$ defined by $\rk_M(S) = \max\{\#I \mid I \subseteq S,\, I \in \I (M) \}$ is the \textit{rank function} of $M$. By the basis exchange axiom, each basis has the same cardinality, which is equal to $\rk_M(E)$ and we call this value the \textit{rank} of the matroid $M$ denoted $\rk(M)$;
        
        \item $\F(M) := \{F \subseteq E \mid \rk(F \cup \{x\}) = \rk(F) + 1 \text{ for all } x \in E \backslash F\}$
        the \textit{flats} of $M$.
        
        \item The \textit{closure} operator $\cl_M \colon 2^E \rightarrow \F(M)$ of $M$ is defined as $\cl_M(S) = \{x \in E \mid \rk_M(S \cup x) = \rk_M(S)\}$. The closure of a subset of $E$ is the smallest flat containing the subset.
        \hfill $\diamond$
    \end{itemize} 
\end{remark}

\begin{definition}[Realisation space of a matroid]
Given a matroid $M$ on $E$ and a field $\kk$, 
a \textit{realisation} of $M$ over $\kk$ is a collection of vectors $\{v_i\}_{i \in E} \subseteq \kk^r$ such that for any subset $I \subseteq E$ we have that $\{v_i\}_{i \in I}$ is linearly independent if and only if $I \in \I(M)$ is an independent set of $M$. Typically, $E = [n]$ and we arrange the vectors $v_i$ as the columns of a matrix $V \in \kk^{r \times n}$. Note, for any realisation, the value $r$ is at least the rank of $M$.
The \textit{realisation space} of a matroid in $\kk^{r \times n}$ is 
    \[
    \Gamma_{M, r} = \{V \in \kk^{r \times n} \mid V \text{ is a realisation of } M\}.
    \]

\end{definition}

\begin{remark}
    The columns of a matrix always give rise to a matroid, but the converse is not true. If a matroid has (resp. does not have) a realisation over a field $\kk$ we call it \textit{realisable} (resp. \textit{non-realisable}) over $\kk$. Throughout the paper, we typically work over $\CC$ so, unless otherwise specified, we will say that a matroid is realisable if it is realisable over $\CC$. \hfill $\diamond$
\end{remark}

\begin{definition}[Matroid variety]\label{matroid variety}
    Let $M$ be a matroid on $[n]$ and $r \ge \rk(M)$ be a positive integer.
    The matroid variety $V_{M, r} = \overline{\Gamma_{M,r}}$ is the Zariski closure in $\kk^{r \times n}$ of the realisation space. We denote by $I_{M,r} = I(V_{M,r}) \subseteq \kk[X]$ the ideal of the matroid variety where $X = (x_{i,j})^{i = 1, \dots, r}_{j = 1, \dots, n}$ is a $r \times n$ matrix of variables.
    If $r$ is fixed, then we write $\Gamma_M$, $V_M$ and $I_M$ for $\Gamma_{M, r}$, $V_{M, r}$ and $I_{M, r}$ respectively. 
\end{definition}

    If a matroid is non-realisable over $\kk$, then its realisation space, and its associated variety are empty.

\begin{definition}[Circuit ideal and basis ideal]\label{construction: circuit and basis ideal} Let $M$ be a matroid on $[n]$ and fix some positive integer $r \ge \rk(M)$. Recall that $\C(M)$ are the circuits of $M$.
Consider the $r \times n$ matrix of variables $X = (x_{i,j})^{i = 1, \dots, r}_{j = 1, \dots, n}$. We define the \textit{circuit ideal} $I_{\C(M)} $ and the \textit{basis ideal} $J_M$ in $R$ as follows:
\[ I_{\C(M)} = \big{\langle} [A|B]_X \, | \, B \in \C(M) \text{ and } A \subseteq [r] \text{ such that } \# A = \# B \big{\rangle} 
\quad \text{and} \quad 
J_M = \Bigg{\langle} \prod_{\substack{B \in \B(M) \\ A \in \binom{[r]}{\rk M}}} [A|B]_X \Bigg{\rangle}. 
\]
In addition, we define the \textit{circuit variety} of $M$ as
    $V_{\C(M)} = V(I_{\C(M)}) \subseteq \kk^{r \times n}$.
\end{definition}

\begin{remark}
    Given a realisation $V \in \kk^{r \times n}$ of $M$, observe that each generator of $I_{\C(M)}$ above vanishes on $V$ by the linear dependence of subsets of columns of $V$. So $I_{\C(M)} \subseteq I_M$. On the other hand, the generator of the principal ideal $J_M$ does not vanish on $V$ as it is a product of non-zero minors of $V$. So the generator of $J_M$ is nowhere-vanishing on $\Gamma_M$. \hfill $\diamond$
\end{remark}

\begin{notation}\label{matroid order}
    Given two matroids $M$ and $N$ on the same ground set $E$,
    we say that $M \leq N$ if $\D(M) \subseteq \D(N)$, i.e., the dependent sets of $M$ are a subset of the dependent sets of $N$. This defines a partial order on the set of matroids with ground set $E$, which we call the \textit{dependency order}. We note that, in the literature, this order is the opposite of the order known as the \textit{weak order}. \hfill $\diamond$
\end{notation}

\begin{definition}[Combinatorial closure of a matroid variety]\label{def:combinatorial_closure}
    Let $M$ be a matroid and fix $r \ge \rk(M)$, we define the \textit{combinatorial closure} of the matroid variety as:
    \[V_{M, r}^{\comb} = \bigcup_{M \leq N} V_{N, r}.\]
    We denote by $I_{M, r}^{comb}$ the ideal $I(V_{M, r}^{comb})$. Whenever $r$ is fixed, then we omit it from the notation.
\end{definition}

We recall the following result from \cite[Lemma 3.5, and Proposition 3.9]{CGMM} that the combinatorial closure is Zariski closed and its defining ideal is the radical of the circuit ideal $\sqrt{I_{\C(M)}}$.

\begin{proposition}\label{ideal matroid variety}
    Let $M$ be a matroid, let $\C(M)$ be the set of its circuits, and fix $r \ge \rk(M)$. 
    \begin{itemize}
        \item The ideal $I_M$ of the matroid variety is the radical of the saturation $I_{\C(M)}$ by  $J_{M}$:

    \[ I_M = \sqrt{\left( I_{\C(M)} : J_{M}^\infty \right)},\] 
    where $\left( I_{\C(M)} : J_{M}^\infty \right) = \{ f \in R \mid \text{for all } g \in J_M \text{ there exists } k \ge 1 \text{ such that } fg^k \in I_{\C(M)}\}.$

    \item The circuit variety of $M$ coincides with the combinatorial closure of the matroid variety:
    \[V_{\C(M)} = V_M^{\comb} \quad \hbox{ and equivalently } \quad I_M^{\comb} = \sqrt{I_{\C(M)}}.\]
    \end{itemize}
\end{proposition}

\subsection{Point-line configurations}

Throughout the paper, we consider matroids of rank at most three as incidence structures referred to as \textit{point-line configurations}. This approach enables us to analyse the realisation space of these matroids using techniques from incidence geometry.

\begin{definition}[Abstract and linear point-line configuration]\label{abstract point-line configuration}
    An \textit{abstract point-line configuration} is a triple $(\PPPP, \LL, \II)$, where the elements of $\PPPP$ and $\LL$ are called \textit{points} and \textit{lines}, respectively. The elements $(p, \ell) \in \II \subseteq \PPPP \times \LL$ are the \textit{incidences}. In this case, we say that \textit{$p$ lies on $\ell$} or, equivalently, that \textit{$\ell$ is incident to $p$}. An abstract point-line configuration is \textit{linear} if there is at most one line incident to a pair of given points, and every line is incident with at least two points.
\end{definition}

Typically, we think of linear point-line configurations that arise from the Euclidean or projective plane. In these cases, the lines are $1$-dimensional affine or linear subspaces, respectively.

\begin{notation}
    For each line $\ell \in \LL$, we identify $\ell$ with $\{p \in \PPPP \mid (p, \ell) \in \II \}$ the set of points lying on the line. So we write $p \in \ell$ whenever $p$ lies on $\ell$ and $\# \ell$ for the number of points on $\ell$. For each $p \in \PPPP$, we denote by $\LL_p = \{\ell \in \LL \mid p \in \ell \}$ the set of lines passing through $p$. \hfill $\diamond$
\end{notation}

Given a realisable rank-three matroid, we define its induced point-line configuration as follows.

\begin{definition}[Point-line configuration of a matroid]
    Let $M$ be a matroid of rank $3$.
    We define the \textit{point-line configuration} $\CCC_M := ([n], \LL_M, \II_M)$ as follows. The point set of $\CCC_M$ is the ground set $[n]$ of $M$. The lines $\LL_M := \{F \in \F(M) \mid \rk(F) = 2, \, |F| \ge 3 \}$ are the rank-two flats of $M$ of size at least three. The incidences $\II_M$ are given by inclusion, i.e., $(p, \ell) \in \II_A$ if and only if $p \in \ell$. Furthermore, if $M$ has a realisation $A \in \kk^{d \times n}$, then we write $\CCC_A$ for the point-line configuration $\CCC_M$.
\end{definition}

In general, the point-line configuration of a realisable matroid is not linear. As the next example shows, this is due to the \textit{loops} and non-trivial \textit{parallel classes} of the matroid.

\begin{definition}[Loops and parallel classes of a matroid]
    Let $M$ be a matroid on ground set $E$. An element $e \in E$ is a \textit{loop} if, for every basis $B$ of $M$, we have $e \notin B$. Equivalently, an element $e \in E$ is a loop if, for every flat $F$ of $M$, we have $e \in F$. The \textit{parallel classes} of $M$ are the rank-one flats. A parallel class is \textit{trivial} if it contains one non-loop element. Equivalently, a parallel class is an inclusion-wise maximal subset of $E$ such that any pair of elements of the subset are dependent.
\end{definition}

\begin{example}
    Consider the matroid $M$ on ground set $[6]$, realised by the columns of the matrix:
    \[
    A = \begin{bmatrix}
        0 & 1 & 2 & 0 & 1 & 0 \\
        0 & 0 & 0 & 1 & 1 & 0 \\
        0 & 0 & 0 & 0 & 0 & 1
    \end{bmatrix}.
    \]
    The first column of $A$ is the zero vector, hence it belongs to no bases of $M$, and so $1$ is a loop of $M$. The parallel classes of $M$ are $\{123, 14, 15, 16\}$. So $123$ is a non-trivial parallel class of $M$. The point-line configuration $\CCC_A$ has lines
    \[
    \LL_A = \{
    12345, 1236, 146, 156
    \}.
    \]
    This configuration is not linear, since the lines $12345$ and $1236$ are incident to $3$ common points. \hfill $\diamond$
\end{example}

We now recall the notion of \textit{simplification} of a matroid, leading to linear point-line configurations.

\begin{definition}[Simple matroid and simplification of a matroid]
    We say that a matroid $M$ is \textit{simple} if $M$ has no loops and every parallel class is trivial. The \textit{simplification} of $M$ is a matroid $M'$ on ground set $E' = \{F \in \F(M) \mid \rk(F) = 1\}$ of rank-one flats of $M$. A set $\{F_1, F_2, \dots, F_k\} \subseteq E'$ is a basis of $M'$ if for each $i \in [k]$ there exists $f_i \in F_i$ such that $\{f_1, f_2, \dots, f_k\}$ is a basis of $M$.
\end{definition}

It is straightforward to check that the simplification of a matroid is indeed a matroid and moreover that it is simple. The simplification of a simple matroid $M$ is itself, and so simplification is a closure operator on the class of matroids. 
In the language of \textit{matroid deletion}, see \cite{Oxl}, the simplification of $M$ can be viewed as the matroid obtained by deleting the loops of $M$ and, for each parallel class of $M$, deleting all but one non-loop element.

\begin{proposition}
    Let $M$ be a simple matroid over $[n]$. Then, the point-line configuration $\CCC_M$ is linear. In particular, if $M$ has a realisation $A \in \CC^{3 \times n}$, then the columns of $A$ are distinct points in $\PP^2_{\CC}$ and the lines of $\CCC_M$ naturally correspond to $1$-dimensional linear subspaces in $\PP^2_{\CC}$.
\end{proposition}

\begin{proof}
    We start the proof by pointing out that:
    \begin{itemize}
        \item singletons in $[n]$ are all the rank-one flats of $M$ and correspond to points of $\PP^2_{\CC}$;
        \item rank-two flats of $M$ of cardinality at least 3 correspond to lines in $\PP^2_{\CC}$.
    \end{itemize}
    Thus, assuming that $M$ is simple, immediately yields that its rank-two flats of cardinality $\geq 3$ contain at least three non-trivial parallel classes. Thus, every line of $\CCC_M$ is incident with at least $3 \geq 2$ points.
    We additionally have to check that any pair of points in $\CCC_M$ lies in at most one line. We prove equivalently that, if $\{pqrs\} \subseteq [n]$ is such that $\rk_M(pqr) = \rk_M(pqs) = 2$, then $\rk_M(pqrs) = 2$. As $M$ is a simple matroid, all 2-subsets of $[n]$ are elements of $\I(M)$. Thus $pqr$ and $pqs$ are circuits of $M$ and the statement follows by the Circuit Elimination Axiom (see e.g. \cite[Lemma 1.1.3]{Oxl}).
\end{proof}

For the rest of the paper we fix the following conventions for working with realisable matroids.

\begin{notation}
    Let $M$ be a realisable matroid on ground set $E$ of rank at most three, and $M'$ its simplification. The point-line configuration $\CCC$ of $M$ is taken to mean $\CCC_{M'}$ the point-line configuration of $M'$. So the points of $\CCC$ are the subsets of $E$ given by the parallel classes of $M$. By a slight abuse of terminology, we say \textit{loops of $\CCC$} for the loops of $M$, and \textit{two points coincide in $\CCC$} when two non-loop elements of $E$ belong to the same parallel class of $M$.\hfill $\diamond$
\end{notation}

Suppose that $M$ and $N$ are matroids of rank at most three on the same ground set $E$. Recall from Notation~\ref{matroid order} the dependency order on matroids. Assume that $M \le N$, i.e., that every dependent set of $M$ is dependent in $N$. In general, the point-line configurations $\CCC_M$ and $\CCC_N$ are different as they have different numbers of points and lines. We introduce the following notation for the purpose of comparing the configurations $\CCC_M$ and $\CCC_N$.

\begin{notation}
    Assume the above setup. We say that a pair of lines $\ell_1$ and $\ell_2$ in $\CCC_M$ \textit{coincide in $\CCC_N$} if the rank-two flats $F_1$ and $F_2$ of $M$, which give rise to $\ell_1$ and $\ell_2$ respectively, are contained in a common rank-two flat of $N$.
    Suppose that $\ell \in \CCC_M$ is a line that arises from a rank-two flat $F$ of $M$. The closure $\cl_N(F)$ gives rise to a collection of points of $\CCC_N$. If these points do not lie on a line of $\CCC_N$, then we say that these points lie on the \textit{ghost line of $\ell$} in $\CCC_N$. \hfill $\diamond$
\end{notation}

We conclude the section by stating some results which correlate the irreducibility of certain matroid varieties with properties of the corresponding point-line configuration.

\begin{proposition}[{\cite[Theorem 4.2]{CGMM}}]\label{six-lines}
    If $M$ is a simple matroid of rank 3, whose point-line configuration $\CCC_M$ has at most six lines, then $V_{M}$ is irreducible with respect to the Zariski topology.
\end{proposition}

\begin{notation}
    Let $\CCC = (\PPPP, \LL, \II)$ be a point-line configuration and $\ell$ a line in $\LL$. We denote as $\CCC \setminus \ell$ the point-line configuration $(\PPPP', \LL', \II')$ where:
    \begin{itemize}
        \item $\PPPP' = \PPPP \setminus \{p \in \PPPP \, | \, (p,\ell) \in \II \hbox{ and } \# \LL_p = 1\}$,
        \item $\LL' = \LL \setminus \{\ell\}$, and
        \item $\II' = \II \setminus \{(q,\ell) \, | \, q \in \PPPP\}$. \hfill $\diamond$
    \end{itemize}
\end{notation}

\begin{theorem}[{\cite[Theorem 4.5]{CGMM}}]\label{irreducibility}
    Let $M$ be a simple rank-3 matroid, and let $\CCC_M$ be its point-line configuration.
    Suppose that $\ell$ is a line of $C_M$ such that $\# \{p \in \ell \mid \#\LL_p \ge 3\} \le 2$.
    Let $M_{\ell}$ be the simple matroid such that $\CCC_{M_{\ell}} = \CCC_M \setminus \ell$. If $\Gamma_{M_{\ell}}$ is irreducible, then so is $\Gamma_{M}$.
\end{theorem}

\section{(Quasi-)liftable configurations}\label{sec:circ}

In this section, we introduce the notions of liftable and quasi-liftable for point-line configurations and prove the associated varieties of such configurations are irreducible. Moreover, we present an irreducible decomposition for their circuit varieties. In particular, given a point-line configuration $\CCC$ with $n$ points, we explore the property that an $n$-tuple of collinear points can be lifted to a non-degenerate realisation of $\CCC$. Typically, this task is highly non-trivial and requires additional conditions on the coordinates of the collinear points.

\medskip
We will first introduce the notion of liftability property for point-line configurations, that play a central role in the irreducibility of the corresponding circuit variety.

\begin{definition}[Liftable configuration]\label{def: liftable} A linear point-line configuration $\CCC$ with $n$ points is called \textit{liftable} over a field $\kk$ if any $n$-tuple of distinct collinear points in $\PP^2_{\kk}$ is the image under a projection of a non-degenerate realisation of $\CCC$. Here, by non-degenerate realisation we mean the datum of an $n$-tuple of points in $\PP^2_{\kk}$, and a bijection with the points of $\CCC$ such that (non-)collinear points in $\CCC$ correspond to (non-)collinear points in $\PP^2_{\kk}$.
\end{definition}

For example, the next result shows that \textit{forest configurations} are liftable over $\CC$. We first recall their definition.
Consider a point-line configuration $\CCC = (\PPPP, \LL, \II)$ whose points $\PPPP$ are endowed with a total order $p_1 < \dots < p_n$. We associate to $\CCC$ the graph $G_{\CCC} = (\PPPP, E)$, where:
$$E = 
\left\{ 
p_ip_j \mid 
p_i, p_j \in \ell \text{ for some } \ell \in \LL,\ p_i < p_j, \text{ and for all } p_k \in \ell \text{ we do not have } p_i < p_k < p_j \right\}.$$
We define the \textit{connected components} of $\CCC$ to be set of sub-configurations corresponding to the connected components of $G_{\CCC}$. We write $\omega$ for the number of connected components of $\CCC$.
The configuration $\CCC$ is called a forest if its graph is a forest. This definition is well posed by \cite[Lemma 5.2]{CGMM}.
    
\begin{lemma}\label{forest lifting}
    Forest configurations are liftable over $\CC$.
\end{lemma}
\begin{proof}Let $\CCC = (\PPPP, \LL, \II)$ be a forest configuration with $n$ points. Starting from $n$ collinear points and a projection centre $P$ in $\PP^2_{\CC}$, one can concretely construct a realisation of $\CCC$ projecting from $P$ to the $n$-tuple of given points. We pick a point $P'\in \PPPP$ such that $\LL_{P'} = 1$ (the existence of such a point is ensured by the forest assumption), and lying on a line $\ell \in \LL$. We associate $P'$ to an arbitrary point in the $n$-tuple, which is the projection image of any point in the line joining itself with the centre $P$. We fix a point on this line and we take a line $\ell'$ through it. By taking the intersections of $\ell '$ with $\# \ell$ fibres of points in the $p$-tuple, $\ell'$ becomes a realisation of $\ell$.

    The points in $\ell'$ projecting to points $Q$, such that $\LL_Q \geq 2$, give rise to branches of the configuration $\CCC$ which do not contain any other point of $\ell$, because of the forest assumption. So, we take arbitrary lines through these points and iterate the argument until the configuration $\CCC$ is fully chased. Figure~\ref{fig:for} shows how this is performed in a sample configuration.
 \begin{figure}[h]
    \centering
    \includegraphics[scale=0.3]{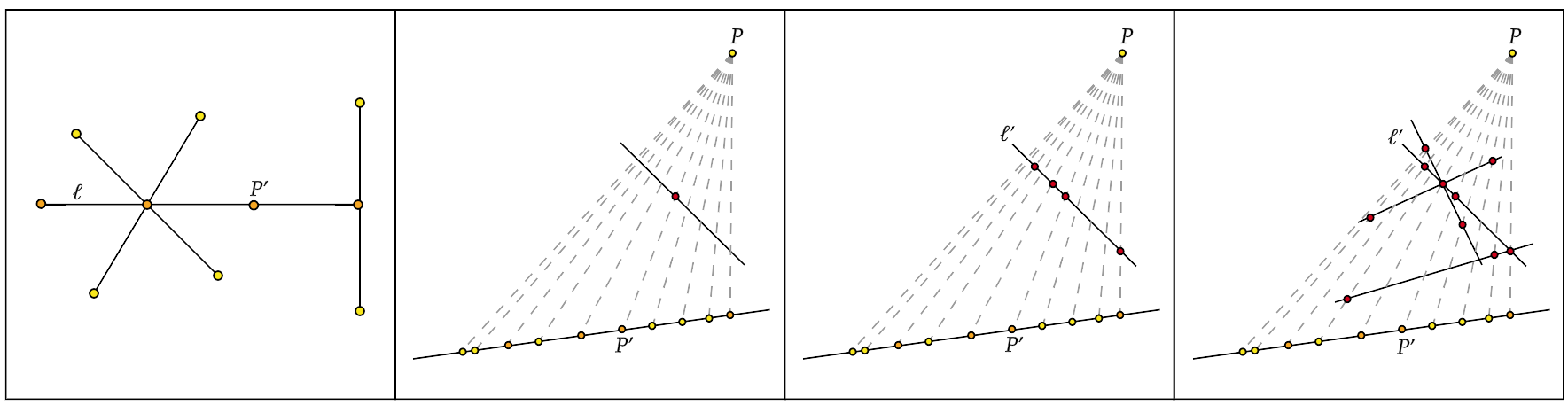}
    \caption{From left to right the figure shows a forest planar configuration with ten points and how it is realised starting from ten collinear points, following the proof of Lemma \ref{forest lifting}.}
    \label{fig:for}
\end{figure}
   \end{proof}

\subsection{Realisation space of liftable configurations}
The liftability problem involves uniquely realisable configurations. Although the general setting can be formulated for all kind of configurations. Now, for any point-line configuration $\CCC$ with $n$ points, we introduce an algebraic tool that plays a crucial role in the liftability problem. In particular, it gives insights about whether $\CCC$ is (quasi-)liftable and the conditions that $n$ collinear points must satisfy so that they may be lifted to a non-degenerate realisation of $\CCC$. For the following construction, we keep in mind the purpose of checking the liftability of a point-line configuration over $\CC$.

\begin{construction}[Collinearity matrix]\label{collinearity matrix}
 Let $\CCC = (\PPPP, \LL, \II)$ be a point-line configuration with $n = \#\PPPP$ points and let $\kk$ be a field. Consider an $n$-tuple of collinear points $P_1, \dots, P_n$ in the projective plane $\PP^2_{\kk}$. Let $P$ be a point that is not collinear with $P_1, \dots, P_n$. After a change of coordinates, we may assume the points lie on the line $z = 0$ and we may write
 \[
    P = \left(\begin{smallmatrix}
    0 \\ 0 \\ 1
    \end{smallmatrix} \right)
    \text{ and }
    P_1 = \left(\begin{smallmatrix}
    x_1 \\ 1 \\ 0
    \end{smallmatrix} \right), \dots, P_n = \left(\begin{smallmatrix}
    x_n \\ 1 \\ 0
    \end{smallmatrix} \right)
    \text{ for some } x_i \in \kk.
\]

The problem is to find $z_1, \dots, z_n \in \kk$ such that the points $\left(\begin{smallmatrix}
    x_i \\ 1 \\ z_i
\end{smallmatrix} \right)$ form a realisation of $\CCC$.

We denote by $[P_iP_j]$ the $2 \times 2$ minor $x_i-x_j$. The \textit{collinearity matrix} $\Lambda$ of $\CCC$ encodes the collinearity conditions imposed by the configuration. The columns of $\Lambda$ are indexed by the points $P_j$ of $\CCC$ for $j \in [n]$, and the rows of $\Lambda$ are indexed by each triple $i = (i_1, i_2, i_3)$, where $P_{i_1}, P_{i_2}, P_{i_3}$ are collinear points. The entries of $\Lambda$ are given by:
\[
(\Lambda)_{i,j} = \begin{cases}
     [P_{i_2}P_{i_3}] & \text{if } j = i_1, \\
    -[P_{i_1}P_{i_3}] & \text{if } j = i_2, \\
     [P_{i_1}P_{i_2}] & \text{if } j = i_3, \\
    0 & \text{otherwise}.
\end{cases}
\]
Suppose that $\ell \subseteq \PPPP$ is a set of collinear points with $\#\ell > 3$. For each $3$-subset of $\ell$, the above construction gives a row of $\Lambda$. These rows are linearly dependent. Moreover, it is straightforward to show that the set of such rows has rank $\# \ell - 2$.
\hfill $\diamond$
\end{construction}

\begin{definition}[Space of lifts, trivial and degenerate liftings]
Let $\CCC$ be a linear point-line configuration with $n$ points. We follow the notation from Construction~\ref{collinearity matrix} for the collinearity matrix $\Lambda$. The matrix $\Lambda$ defines the linear system:
\begin{equation}\label{lifting system} \Lambda \begin{pmatrix} z_1 \\ \vdots \\ z_n \end{pmatrix} = \begin{pmatrix} 0 \\ \vdots \\ 0 \end{pmatrix}.\end{equation} 
The solution space of~\eqref{lifting system}, denoted as $\LLL_{\CCC}$, is the \textit{space of lifts} for $\CCC$.

For each $z \in \LLL_{\CCC}$, we obtain a \textit{lifted configuration} of points $P_i$ with coordinates  $\left(\begin{smallmatrix}
    x_i\\ 1 \\ z_i
\end{smallmatrix}\right)$.

By construction, these points realise all the collinearity conditions of $C$. We call $z$ \textit{trivial} if these lifted points are collinear. We say $z$ is a \textit{degenerate lift} if these lifted points are not a realisation of $\CCC$ but do not lie on a single line.
\end{definition}

\begin{lemma}\label{dimesion space of lifts}
    The linear system~\eqref{lifting system} has non-trivial solutions if and only if $\dim(\LLL_{\CCC}) \geq 3$.
\end{lemma}
\begin{proof}
    Trivial liftings of the $n$ points contribute 2 dimensions to the dimension of $\LLL_{\CCC}$. Therefore, a non-trivial lift exists if and only if the solution space has dimension at least three.
\end{proof}
  
This leads to an efficient method to check if an arrangement is not liftable.

\begin{example}
 By applying Lemma~\ref{dimesion space of lifts} to the collinearity matrices of the quadrilateral set and the $3 \times 4$ grid, in the proofs of Theorems~\ref{TFAE quad set} and~\ref{TFAE 3x4 grid}, we observe that they are not liftable.
\hfill $\diamond$
\end{example}
 
For the following family of point-line configurations, the fact that System~\eqref{lifting system} has solution space of dimension $\geq 3 \cdot \omega$ completely solves the liftability problem. 

\begin{definition}[Maximal matroid]
    Let $M$ be a simple matroid of rank $3$, whose configuration $\CCC_M$ has no triplets of concurrent lines. We say that $M$ is \textit{maximal} if it is maximal, with respect to the dependency order, among the realisable simple matroids of rank $3$ with no triplets of concurrent lines.
\end{definition}

\begin{example}\label{example:non-maximal}
    The underlying simple matroids of the quadrilateral set and the $n \times m$ grid configurations are maximal. This can be deduced by how the property of being maximal reflects on the point-line configuration associated to the matroid.
    In particular, the fact that there are no simple rank-three matroids $N > M$ over $[n]$, with $\CCC_N$ having no triplets of concurrent lines, is equivalent to the fact that there is no linear point-line configuration $\CCC' \neq \CCC_{M}$, with $n$ different points, satisfying conditions:
    \begin{itemize}
        \item all the collinearities (i.e., triplets of collinear points) of $\CCC_M$ are collinearities of $\CCC'$;
        \item the points of $\CCC'$ do not lie on a single line.
    \end{itemize}
    In other words, the underlying simple matroids of the quadrilateral set and the $n \times m$ grid configurations are maximal because any non-trivial lifting of them in $\PP^2_{\CC}$ is a realisation of the matroid.
    
    Differently, the simple matroid $M$ over $[12]$, whose circuits are $\C(M) = \min\left(\Delta \cup \binom{[12]}{4}\right)$, where $\Delta = \binom{[6]}{3}\cup \{178, 289, 379, 4\,10\,11, 5\,10\,12, 6\,11\,12\}$, is not maximal. Indeed, one can consider, for instance, the simple matroid $M_1$ over $[12]$, whose circuits are $\C(M_1) = \min\left(\Delta_1 \cup \binom{[12]}{4}\right)$, where $\Delta_1 = \binom{[9]}{3}\cup \{4\,10\,11, 5\,10\,12, 6\,11\,12\}$. As $\C(M) \subsetneq \C(M_1)$, we have that $M_1 > M$. In particular, $M$ is contained in the two maximal matroids $M_1$ and $M_2$, whose configurations are depicted in Figure~\ref{fig:non-maximal simple}. Their maximality can be verified by the method mentioned at the beginning of the Example. In particular, both of them have triplets of non-collinear points but they are not a realisation of $M$. Adding further dependencies on $M_1$ and $M_2$ would lead to a non-simple matroid or to a rank drop (see matroid $M_3$, in Figure~\ref{fig:non-maximal simple}).
        \begin{figure}[h]
        \centering
        \includegraphics[scale=0.4]{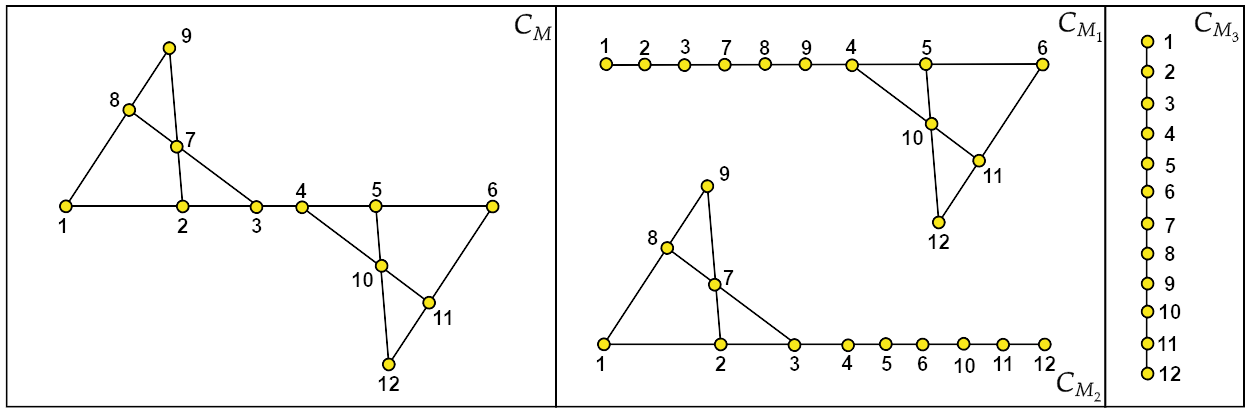}
        \caption{In the notation of Example~\ref{example:non-maximal}, from left to right there are the linear point-line configurations of the matroids $M$, $M_1$, $M_2$, and $M_3$.}
        \label{fig:non-maximal simple}
    \end{figure}
    \hfill $\diamond$
\end{example}

\begin{lemma}\label{characterization liftability}
    Let $\CCC = (\PPPP, \LL, \II)$ be a linear point-line configuration with $n$ points, with  connected components $\CCC_1, \dots, \CCC_{\omega}$. Assume that $\CCC_1, \dots, \CCC_{\omega}$ derive from a maximal matroid. Then, the following statements are equivalent.
    \begin{itemize}
        \item[$\bullet$] The configuration $\CCC$ is liftable.
        \item[$\bullet$] The linear system $\Lambda z = 0$, where $\Lambda$ is the collinearity matrix of $\CCC$, has solution space of dimension at least $3\cdot \omega$, or equivalently $n -3\cdot \omega \geq \rk \Lambda$.
    \end{itemize}
\end{lemma}

\begin{proof}
    Follows from Lemma~\ref{dimesion space of lifts}.
\end{proof}

\begin{example}\label{lifting 3x3}
Let $1, \dots, 9$ be collinear points in the projective plane. Then, it is always possible to construct a realisation of a $3\times3$ grid such that the 9 points are a projective image of the grid itself. In other words, $3 \times 3$ grids are liftable.
By assuming the incidence structure in Figure~\ref{fig:3x3g}, the collinearity matrix of a $3 \times 3$ grid is:
$$ \Lambda = \left( \begin{matrix}
[23] & -[13] & [12] & 0 & 0 & 0 & 0 & 0 & 0\\
[47] & 0 & 0 & -[17] & 0 & 0 & [14] & 0 & 0\\
0 & [58] & 0 & 0 & -[28] & 0 & 0 & [25] & 0\\
0 & 0 & [69] & 0 & 0 & -[39] & 0 & 0 & [36] \\
0 & 0 & 0 & [56] & -[46] & [45] & 0 & 0 & 0\\
0 & 0 & 0 & 0 & 0 & 0 & [89] & -[79] & [78]
\end{matrix} \right)$$
As a consequence, the linear system $\Lambda (z_1 \dots z_9)^t = (0 \dots 0)^t$ has a solution space of dimension at least 3. In other words, it is possible to choose $z_1, \dots, z_9$ such that
$$\begin{cases}
[23] \cdot z_1 -[13] \cdot z_2 + [12] \cdot z_3 = [123] = 0\\
[47] \cdot z_1 -[17] \cdot z_4 + [14] \cdot z_7 = [147] = 0\\
[58] \cdot z_2 -[28] \cdot z_5 + [25] \cdot z_8 = [258] = 0 \\
[69] \cdot z_3 -[39] \cdot z_6 + [36] \cdot z_9 = [369] = 0 \\
[56] \cdot z_4 -[46] \cdot z_5 + [56] \cdot z_4 = [456] = 0\\
[89] \cdot z_7 -[79] \cdot z_8 + [78] \cdot z_9 = [789] = 0
\end{cases}$$
and such that the points $\left(\begin{smallmatrix}
    x_i \\ 1 \\ z_i
\end{smallmatrix} \right)$ for $i = 1, \dots, 9$ span the whole projective plane. In turn, by Lemma~\ref{characterization liftability}, there exists a non-degenerate $3 \times 3$ grid whose image via the projection through $P$ on the line $r$ consists exactly of points $1, \dots, 9$. \hfill $\diamond$
\end{example}

\begin{remark}
    The liftability property is preserved by certain operations on point-line configurations. If $\CCC$ is liftable, then $\CCC \setminus \ell$ (if realisable) is still liftable for any $\ell \in \LL$. In addition, adding a point to a line (with no extra collinearity requirement) does not affect liftability. In this case, both the number of points and the rank of the collinearity matrix increase by one. \hfill $\diamond$
\end{remark}

\begin{lemma}\label{perturbation}
    Let $M$ be a simple matroid of rank 3 over $[n]$, whose associated point-line configuration $\CCC_M$ is liftable. Then, any realisation of the 2-uniform matroid over $[n]$, in $\CC^{3n}$, is an element of the matroid variety $V_M$. 
\end{lemma}
\begin{proof}
    Let $C \in \CC^{3n}$ be a realisation of the 2-uniform matroid over $[n]$. Then the coordinates $C_{1,1}, \dots, C_{3,n}$ of $C$ can be seen as the $(x,y,z)$-coordinates of $n$ collinear points in the projective plane. These can be represented with the $3 \times n$ matrix below, up to performing a change of coordinates in $\PP^2_{\CC}$ - as in Construction~\ref{collinearity matrix} - which yields, in turn, a change of coordinates on $\CC^{3n}$.
    \[ C = \left( \begin{matrix} C_{1,1} & C_{1,2} & \dots & C_{1,n} \\
        C_{2,1} & C_{2,2} & \dots & C_{2,n} \\
        C_{3,1} & C_{3,2} & \dots & C_{3,n}
    \end{matrix} \right) = \left( \begin{matrix} x_1 & x_2 & \dots & x_n \\
        1 & 1 & \dots & 1 \\
        0 & 0 & \dots & 0
    \end{matrix} \right).\]
        The liftability of $\CCC_M$ ensures the existence of $z_1, \dots, z_n \in \CC$, such that the points of coordinates $(x_i \, 1 \, z_i)^t$, for $i = 1, \dots, n$, are a non-degenerate realisation of $M$. Now, let $\varepsilon$ be an arbitrary positive real number and set $z:= \max_{i = 1, \dots, n} |z_i|$. Let $D \in \CC^{3n}$ be the point of coordinates: 
    \[ D = \left( \begin{matrix} x_1 & x_2 & \dots & x_n \\
        1 & 1 & \dots & 1 \\
        \frac{\varepsilon}{nz}z_1 & \frac{\varepsilon}{nz}z_2 & \dots & \frac{\varepsilon}{nz}z_n
    \end{matrix} \right).\]
    The point $D$ is still a realisation of $M$ because $\Gamma_M$ is a semi-algebraic set defined by the (non-)vanishing of determinants of $3 \times 3$ matrices whose columns are $(x,y,z)$-coordinates of points in the projective plane. In our case, by multilinearity of determinants:
    $$ \left[ \begin{matrix} x_i & x_j & x_k \\
        1 & 1 & 1 \\
        \frac{\varepsilon}{nz}z_i & \frac{\varepsilon}{nz}z_j & \frac{\varepsilon}{nz}z_k
    \end{matrix} \right] = \frac{\varepsilon}{nz} \left[ \begin{matrix} x_i & x_j & x_k \\
        1 & 1 & 1 \\
        z_i & z_j & z_k
    \end{matrix} \right], \qquad \hbox{for any } 1 \leq i < j < k \leq n.$$
    As the left-hand-side vanishes if and only if the right-hand-side vanishes, this proves that any Euclidean open subset of $\CC^{3n}$ containing $C$ intersects $\Gamma_M$. In turn, this implies that any Zariski open subset containing $C$ intersects $\Gamma_M$, thus $C \in \overline{\Gamma_M} = V_M$.
\end{proof}

\begin{definition}[Quasi-liftable configuration]~\label{def: quasi-liftable} A linear point-line configuration $\CCC = (\PPPP, \LL, \II)$ is called quasi-liftable if $\CCC$ is not liftable but, $\CCC \setminus \ell$ is liftable for every $\ell \in \LL$.
\end{definition}

\begin{example}
Our configurations of interest, namely the quadrilateral set and the $3 \times 4$ grid, are quasi liftable (see Figures~\ref{fig:qs} and~\ref{fig:3x4g}). However, when a line is added to a quasi liftable configuration, the quasi-liftability property is not preserved. \hfill $\diamond$
\end{example}

\begin{figure}

\centering
\hspace{30pt} \subfloat[$3 \times 3$ Grid]{\label{fig:3x3g}
\centering
\includegraphics[scale = 0.3]{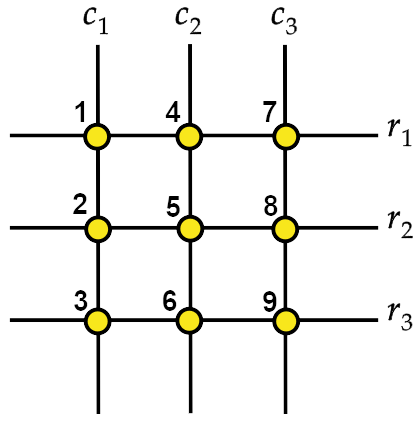}
}
\hfill
\subfloat[Quadrilateral set]{\label{fig:qs}
\centering
\includegraphics[scale = 0.3]{quadset.png}
}
\hfill
\subfloat[$3 \times 4$ Grid]{\label{fig:3x4g}
\centering
\includegraphics[scale = 0.3]{3x4grid.png}
} \hspace{30pt}
\caption{Examples of liftable and quasi-liftable plane arrangements.}
\label{fig:exa}
\end{figure}

\subsection{Irreducible decomposition of varieties of (quasi-)liftable configurations}
\label{sec:decompose}
In this section, we prove the irreducibility of the matroid varieties of liftable and quasi-liftable point-line configurations, assuming there are no triplets of concurrent lines in the configuration. Specifically, in Theorem~\ref{decomposition liftable}, we establish the irreducibility of the matroid variety in both cases. However, for the circuit variety, we demonstrate that while it is irreducible for the liftable configuration (Theorem~\ref{circuit variety of liftable conf.}), the circuit variety of the quasi-liftable matroid has two irreducible components (Corollary~\ref{decomposition quasi-liftable}).

\medskip 
In this section, unless stated otherwise, we assume that all configurations are for simple matroids of rank three, and
   preserve their realisability when considering a sub-configuration.

\medskip
We first prove the irreducibility of the matroid and circuit varieties of liftable configurations.

\begin{theorem}\label{decomposition liftable}
    Let $M$ be a matroid, whose associated point-line configuration $\CCC_M$ has no triplets of concurrent lines. Then, the matroid variety $V_{M}$ is irreducible.
\end{theorem}

\begin{proof}
If $\CCC_M$ has less than 6 lines, then $V_M$ is irreducible by  Proposition~\ref{six-lines}. 
Assume that $\CCC_M$ has more than 6 lines.  Consider a sub-configuration $\CCC' = \CCC_M \setminus \{\ell_1, \dots,\ell_{\# \LL_M - 6} \}$, for $\{\ell_1, \dots,\ell_{\# \LL_M - 6} \} \subseteq \LL_M$. Now, let $M'$ be the corresponding matroid. Then, $\Gamma_{M'}$ is irreducible with respect to the Zariski topology, as $V_{M'}$ is irreducible by Proposition~\ref{six-lines} and $\Gamma_{M'}$ is Zariski dense in $V_{M'}$.
Note that $\# \LL_p \leq 2$ for any $p \in \PPPP_M$ by assumption.    Let us now consider $\CCC'' = \CCC_M \setminus \{\ell_1, \dots,\ell_{\# \LL_M - 7} \}$. Among the intermediate configurations between $\CCC'$ and $\CCC_M$, $\CCC''$ is such that $\CCC' = \CCC'' \setminus\ell_{\# \LL_M - 6}$ (denote as $M''$ the corresponding matroid). Theorem $\ref{irreducibility}$ implies that $\Gamma_{M''}$ is irreducible and so is its Zariski closure $V_{M''}$. If $\CCC'' = \CCC_M$ we are done. If not, we introduce a configuration $\CCC'''$ and we argue analogously. The process terminates after $\# \LL_M - 6$ steps. 
\end{proof}

 \begin{remark}
        In view of Lemma~\ref{forest lifting}, Theorem~\ref{decomposition liftable} generalises
        Proposition 5.9, and Theorem 5.11 from \cite{CGMM}. Thereby an analogous result was proved, for
        forest-configurations.
        \hfill $\diamond$
    \end{remark}
    
Recall that every matroid variety $V_M$ is equal to its combinatorial closure $V_M^{\comb}$, but not necessarily to the circuit variety $V_{\mathcal{C}(M)}$, see and Definition~\ref{def:combinatorial_closure} and Proposition~\ref{ideal matroid variety}. In the following theorem, we demonstrate that for liftable configurations, all these varieties coincide. We first establish the theorem and subsequently prove the technical lemma, Lemma~\ref{non-simple}, used in the proof.

\begin{theorem}\label{circuit variety of liftable conf.}
    Let $M$ be a matroid whose associated point-line configuration $\CCC_M$ is liftable and has no triplets of concurrent lines. Then,
    
    $$V_{\C(M)} = V_M \qquad \hbox{ and equivalently } \qquad \sqrt{I_{\C(M)}} = I_{M}.$$
\end{theorem}

\begin{proof}
    We prove the result by induction on the number of lines in $\CCC_M$.
     
        If $\CCC_M$ has no lines, then $M$ is the 3-uniform matroid. In this case, the realisation space $\Gamma_M$ is a Zariski open subset, and $V_M = \overline{\Gamma_M} = \CC^{3n} = V_{\C(M)}$.
    
    Assume that the result is true for any matroid fulfilling the hypotheses with $\leq m$ lines. Consider a configuration $\CCC_M$ with $m+1$ lines and let $C$ be a point in $V_{\C(M)} \setminus \Gamma_M$, which equals to $V_{M}^{\comb} \setminus \Gamma_M$, by Proposition~\ref{ideal matroid variety}. We want to prove that $C \in V_M$.
    If $M$ is a matroid over $[n]$, the point $C \in \CC^{3n}$ realises a matroid $N \ge M$, over $[n]$, which is dependent for $M$.
    In particular, the coordinates of $C$ can be represented as a matrix whose columns are the $(x,y,z)$-coordinates of $n$ points in the projective plane.
    As $\CCC_M$ is liftable (Lemma~\ref{non-simple}), we restrict to the case where $N$ is a simple matroid.

    \medskip

    If the associated point-line configuration $\CCC_N$ consists of collinear points, the result holds by Lemma~\ref{perturbation}. Otherwise, for any line $\ell \in \LL_M$, there is a well-defined projection $\pi_{\ell} \colon \CC^{3n} \to \CC^{3k}$, for some $k \leq n$, which deletes from the matrix $C$ the columns that are coordinates of points in $\CCC_M$, but not in $\CCC_M \setminus \ell$.
    Now, for any line $\ell$, let us denote $M_{\ell}$ for the simple matroid whose point-line configuration is $\CCC_M \setminus \ell$. Then, $\pi_{\ell}(C) \in V_{M_{\ell}}$ because $\CCC_M \setminus \ell$ is liftable. Thus,  $C \in \bigcap_{\ell \in \LL} \pi_{\ell}^{-1}(V_{M_{\ell}})$.
    To conclude, we prove that $\bigcap_{\ell \in \LL} \pi_{\ell}^{-1}(V_{M_{\ell}}) \subseteq V_{M}$. This is equivalent to the inclusion:
    $$\left( \bigcap_{\ell \in \LL} \pi_{\ell}^{-1}(V_{M_{\ell}}) \right) ^c \supseteq \left( V_{M} \right)^c = \left( \overline{\Gamma_{M}}\right)^c = (\Gamma_{M}^c)^{\circ}.$$
    Now, as $M$ is not 3-uniform, $(\Gamma_{M}^c)^{\circ}$ is not empty. Let $D = (D_{1,1}, \dots, D_{3,n})$ be in the Zariski interior of $\Gamma_{M}^c$. We show that there exists a line $\ell$ such that $\pi_{\ell}(D) \notin V_{M_{\ell}}$. Now, $\Gamma_{M}$ is the semi-algebraic set defined by $p_i (x_{1,1}, \dots, x_{3,n}) = 0$ and $q (x_{1,1}, \dots, x_{3,n}) \neq 0$ for certain homogeneous polynomials $p_i$ and $q$ imposing, respectively, the dependence and independence relations for the matroid $M$.

    \medskip
    
    If one of the polynomials $p_i$ does not vanish when evaluated at the coordinates of $D$, then there is a collinearity of $\CCC_M$, which is not satisfied by $D$. Let $\ell$ be any line in $\LL_M$ not requiring the critical collinearity, then $\pi_{\ell}(D) \notin V_{M_{\ell}}$. Hence, $C \in \bigcap_{\ell \in \LL} \pi_{\ell}^{-1}(V_{M_{\ell}}) \subseteq V_{M}$.

    \medskip
    
    The only excluded case is when, for all $i$, $p_i(D_{1,1}, \dots, D_{3,n}) = q(D_{1,1}, \dots, D_{3,n}) = 0$. We show that this cannot happen; specifically we prove that $D$ is a point of $V_M$. On the one hand, the vanishing of the $p_i$'s implies that $D \in V_{C(M)} \setminus \Gamma_{M}$. On the other hand, as long as $D$ is in the interior of $\Gamma_{M}^c$, there exists a Zariski open subset entirely contained in $\Gamma_{M}^c$ and containing $D$.

    Now, Lemma~\ref{non-simple} allows us to restrict to the case where $D$ realises a simple matroid, dependent for $M$, and Lemma~\ref{perturbation} excludes the possibility that the $n$ points given by the realisation of $D$ are collinear. In other words, the point-line configuration realised by $D$ has $n$ distinct points, among which at least three of them are not collinear, and respects all the collinearities of $M$. There are some extra-collinear points which are the projective image of a sub-configuration of $\CCC_M$. As $\CCC_M$ is  liftable, all its sub-configurations are liftable as well. Thus, the unwanted collinearity can be resolved by a small arbitrary lifting.   
    
    To conclude,  if $D$ satisfies $p_i(D_{1,1}, \dots, D_{3,n}) = q(D_{1,1}, \dots, D_{3,n}) = 0$, then any Euclidean open subset containing $D$ intersects $\Gamma_M$, contradicting that $D$ is in the Zariski interior of $\Gamma_M^c$.
    \end{proof}

We now prove the technical lemma used in the proof of Theorem~\ref{circuit variety of liftable conf.}.
\begin{lemma}\label{non-simple}
    Let $M$ be a matroid whose point-line configuration $\CCC_M$ is liftable and has no triplets of concurrent lines. Let $C \in \CC^{3n}$ be a point in $V_{\C(M)} \setminus \Gamma_M$, realising a non-simple matroid $N > M$. Then, for any $\varepsilon > 0$, there exists $C'$ in the Euclidean open ball $B(C,\varepsilon)$ realising a simple matroid $N'$ which is dependent for $M$, or equivalently,  is $M \leq N'$.
    Further, if $C' \in V_M$, then $C \in V_M$.
\end{lemma}

\begin{proof} The coordinates of the point $C$ can be represented by the $3 \times n$ matrix:
    $$C = \left( \begin{matrix} C_{1,1} & C_{1,2} & \dots & C_{1,n} \\
        C_{2,1} & C_{2,2} & \dots & C_{2,n} \\
        C_{3,1} & C_{3,2} & \dots & C_{3,n}
    \end{matrix} \right) = (c_1 \ c_2 \  \dots \  c_n ),$$
    where $c_i \in \CC^3$ for each $i$.
    Since the matroid $N$, realised by $C$, is non-simple, the matrix above might contain zero vectors $c_i = 0$ or linearly dependent pairs of columns $c_i = \lambda \cdot c_j$ for some $\lambda \in \CC^*$. To prove the result, we proceed by induction on $n$. Since the base case is trivial, we may assume that the result holds for any sub-configuration of $\CCC_M$, having less than $n$ points. We recall that the point-line configuration $\CCC_N$ of $N$ is associated to the simplification of the matroid.

    Fixed a positive real number $\varepsilon$, the procedure to obtain the point $C'$ will consist in perturbing the columns of $C$ finitely many times by adding to them vectors whose norm is bounded by $\varepsilon$. For the choice of the base-field $\CC$, and the continuity property of the Euclidean distance over $\CC^{3n}$, the value of $||C - C'||$ will eventually be bounded by a continuous function of $\varepsilon$.
    
    By assumption we have that $N > M$, so every flat of rank $2$ in $M$ is dependent in $N$. This means that for every line $\ell$ of $\CCC_M$, the points $c_\ell := \{c_i : i \in \ell\}$ lie in a $2$-dimensional linear subspace $\hat \ell \subseteq \CC^3$, which we call the \textit{ghost line} of $\ell$. There are two possibilities for the points $c_\ell$: 
    \begin{itemize}
        \item either their linear span is $2$-dimensional, in which case $\hat \ell$ is uniquely determined; or
     \item the linear span has dimension strictly less than $2$, in which case we choose $\hat \ell$ generically such that it contains $c_\ell$. 
     \end{itemize}
     We construct the point $C'$ algorithmically by performing the following steps.
    
    \medskip
    
 \noindent   \textbf{Loops:} Suppose there is a zero column $c_i = 0$ of $C$. By assumption, the point $i$ of $\CCC_M$ lies on at most two lines $\ell_1$ and $\ell_2$ (if this is not the case, then the same strategy can be performed with fewer constraints). Let $v \in \hat {\ell_1} \cap \hat {\ell_2}$ be a point in the intersection of the two ghost lines. Without loss of generality, we may assume that $|v| = 1$. We perturb $c_i$ by moving it to the point $\varepsilon' v$. This results in a configuration $N'$ with $N'\geq M$. 

 \smallskip
 From now on, we assume that all points $c_i$ are nonzero.
    
    \medskip
    
  \noindent     \textbf{Multiple points:} Let $\mathscr{I} = \{I_1, \dots, I_K\}$ be the set of multiple points of $\CCC_N$, i.e., the rank-one flats of $N$ of cardinality strictly higher than one. For any $I_k \in \mathscr{I}$ let $\mathscr{I}_k = \{i_1, \dots, i_{n_k}\}$ be the corresponding set of points in $\CCC_M$ such that $\Span_{\CC}\{c_{i_1}, \dots, c_{i_{n_k}}\}$ is one-dimensional. We introduce a procedure to perturb these points so that they realise the point-line configuration of a simple matroid, dependent on $M$. We address cases based on the $i_l$'s. Initially, we handle points lying on no lines of $\CCC_M$ (step~S1) and points lying on one line of $\CCC_M$ (step S2). At this stage, all multiple points can be assumed to consist of points belonging to two lines of $\CCC_M$. Now, either all (ghost) lines through a multiple point coincide or not. In the former case, we proceed as in step S3;~in the latter case, we apply~steps~S4~and~S5.

        \begin{itemize}
            \item[S1.] For any index $k = 1, \dots, K$ and for any $l = 1, \dots, n_k$, if the point $i_l \in \mathscr{I}_k$ does not belong to any line in $\CCC_M$, then it can be translated along any direction. In other words, we perturb the column $c_{i_l}$ by adding to it a vector $\varepsilon v$, where $v$ has unitary norm (see Figure~\ref{fig:fat IV} for an example). In turn, we can assume that for any $k = 1, \dots, K$, all points in $\mathscr{I}_k$ lie at least on a line in $\CCC_M$.

                \begin{figure}[h]
                \centering
                \includegraphics[scale=0.31]{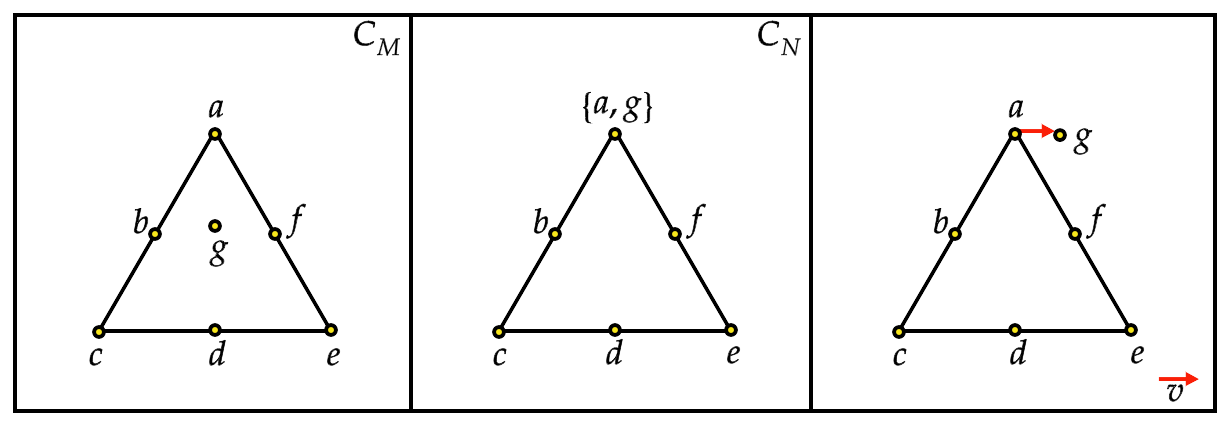}
                \caption{The figure refers to step S1, and illustrates how to solve the double point $\{a,g\}$. In view of the fact that the point $g$ does not lie on any line of $\CCC_M$, it can be translated along the direction $v$.}
                \label{fig:fat IV}
                \end{figure}

            \item[S2.] For any $k = 1, \dots, K$ and for any $l = 1, \dots, n_k$, we check whether the point $i_l \in \mathscr{I}_k$ belongs to a single line $\ell$ in $\CCC_M$. If this is the case, the point $i_l$ can be realised by translating the corresponding column $c_i$ along the ghost line $\hat \ell$, as in Figure~\ref{fig:fat V}. Thus, we can further assume that for any $k = 1, \dots, K$, all points in $\mathscr{I}_k$ lie at the intersection of two lines in $\CCC_M$.

                \begin{figure}[h]
                \centering
                \includegraphics[scale=0.31]{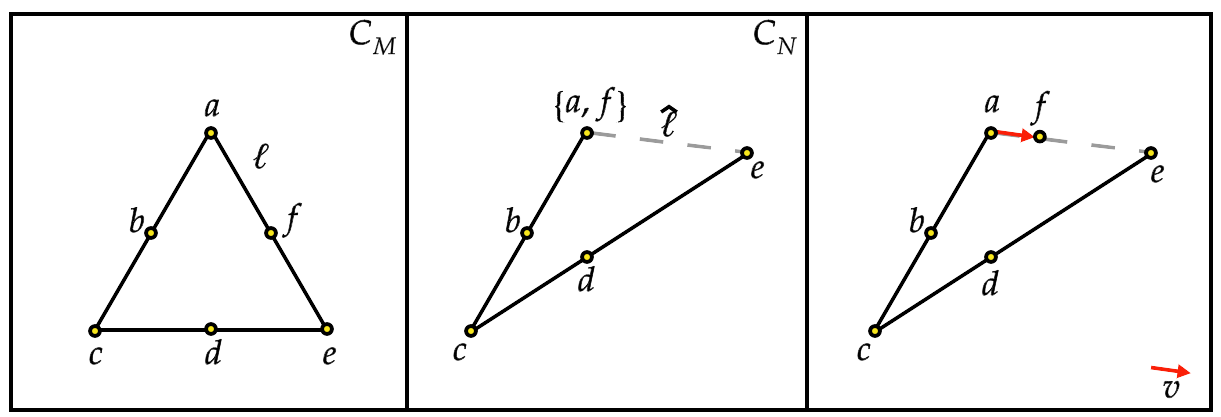}
                \caption{The figure illustrates step S2, and how to solve the double point $\{a,f\}$.~In~the starting configuration $\CCC_M$, the point $f$ lies solely on the line $\ell$, and can be translated along~the~direction~$v$.}
                \label{fig:fat V}
                \end{figure}
          
            \item[S3.]  
For any $k$, if all the (ghost) lines through $I_k$ coincide on a unique (ghost) line $\ell$, then we perturb all points of $\mathscr{I}_k$ along $\ell$. Figure~\ref{fig:fat VIII} illustrates that this operation might not lead to a full realisation of $\CCC_M$, as it does not take the independence relations into account. However, it generates all required collinearities, leading to a dependent matroid for $M$. Any remaining multiple point now lies at the intersection of the realisation of at least two non-overlapping lines.

                \begin{figure}[h]
                \centering
                \includegraphics[scale=0.31]{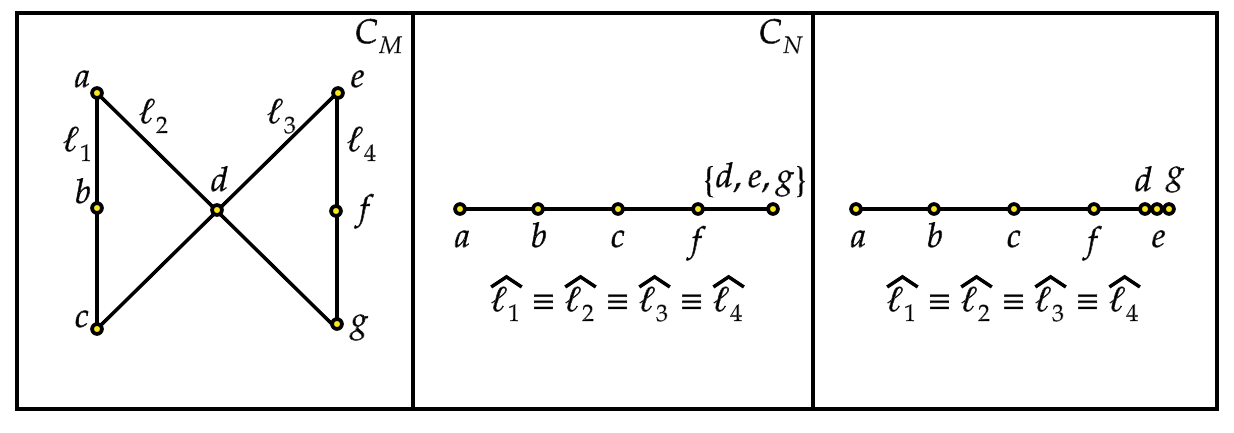}
                \caption{The realisation of lines $\ell_1, \dots, \ell_4$ of $\CCC_M$ all coincide in $\CCC_N$. Therefore, following step S3, the triple point $\{d,e,g\}$ can be split along the line $\hat{\ell_1} \equiv \dots \equiv \hat{\ell_4}$ to realise all collinearities of $\CCC_M$.}
                \label{fig:fat VIII}
                \end{figure}

            \item[S4.] We now remove all remaining multiple points from the configuration, and we consider the multiple lines $L$ that are incident to them. We resolve the multiple lines using a lifting procedure, which can be followed graphically from Figure~\ref{fig:fat VI}. We consider the following two steps:

                    \begin{itemize}
                    \item[S4.1.] In view of the liftability property of $\CCC_M$ and the induction hypothesis, we can perform a lifting of the sub-configuration of $N$ that involves only the remaining points on $L$ and the (ghost) lines that coincide with $L$. Here, we pick our projection point in general position, away from all lines generated by points of $\CCC_N$. It follows from the proof of Lemma~\ref{perturbation} that the distance between $C$ and the lifted configuration can be bounded by $\varepsilon$.

                    \item[S4.2.] If a lifted point $i$ was the intersection of $L$ with another (ghost) line $\hat \ell$, prior to the lifting, then we need to further perturb it. In $\CCC_M$, the point $i$ lies at the intersection of two lines $\ell$ and $\ell'$. In $\CCC_N$, $\hat \ell'$ coincides with the line $L$. After the lifting, we have that $\hat \ell'$ is lifted and intersects $\hat \ell$ at a unique point. We redefine $i$ to be this unique intersection point, whose distance from the previous lifted point is a continuous function of $\varepsilon$.
                    \end{itemize}
            
            For any lifted (ghost) line resulting from this operation and containing a point in one of the multiple points through $\ell$, we can assume, without loss of generality, that it is still arbitrarily close to the concerned multiple points. This can be explained considering that if the multiple points had not been removed from the configuration, the lifting could still have been performed. In particular, any such lifted line intersects all the other lines incident to the multiple point in a neighborhood of the multiple point itself.

                \begin{figure}[h]
                \centering
                \includegraphics[scale=0.3]{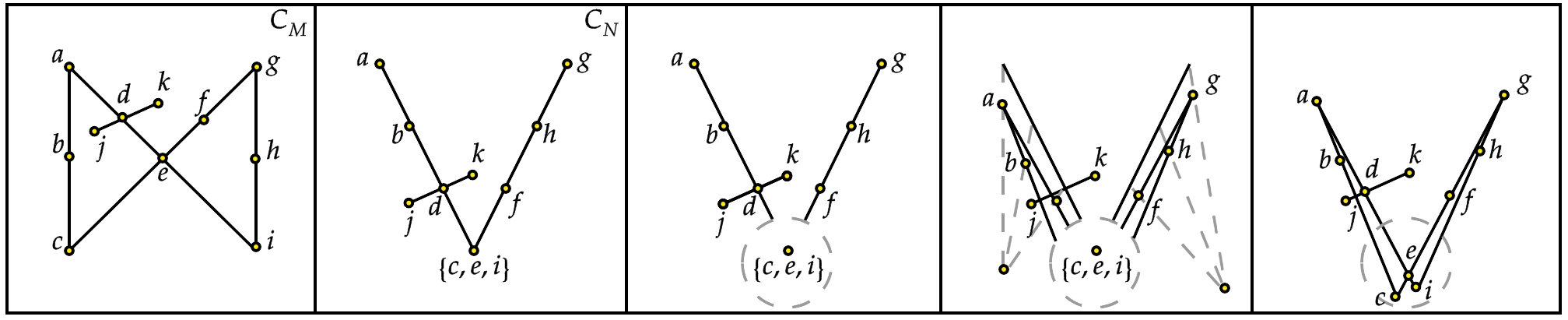}
                \caption{The figure demonstrates step S4 in action from left to right. Here, we address a triple point through liftings. Both lines incident to ${c,e,i}$ in $\CCC_N$ realise two rank-two flats of $M$. Therefore, a lifting procedure is necessary to resolve the multiple point. In the final square, we illustrate how to intersect the lifted lines to complete the configuration.}
                \label{fig:fat VI}
                \end{figure}

            \item[S5.] For any multiple point $I_k \in \mathscr{I}$, we consider the lines through $I_k$ excluded from step S4. Namely, those corresponding to rank-two flats of $N$ that do not contain more than a rank-two flat of $M$. Let $m_k$ be the number of such lines. If $m_k \geq 3$, we have to perturb $m_k-2$ of them so that they are not all concurrent on the same point in the projective plane. This can be done considering that none of these lines is a multiple line, and for any such line $\ell$, we perform the following two steps, depicted in Figure~\ref{fig:fat VII}.

            \begin{itemize}
                \item[S5.1.] We pick a generic direction $v \in \CC^3$. Our goal is moving all points on $\ell$ in the direction of $v$. 
                \item[S5.2.] We select a basis $B$ for the (ghost) line $\ell$. Then, fixing a unitary vector $v$, we consider the linear subspace $\mu = \operatorname{Span}_{\mathbb{C}}\{b + \varepsilon v : b \in B\}$. For each point $i$ lying on $\ell$, we perturb $i$ as follows. By assumption, $i$ belongs to exactly one or two lines of $\CCC_M$. If $i$ belongs to two lines $\ell$ and $\ell'$, then we move $i$ to the closest point in the intersection of the ghost line $\hat{\ell'}$ and $\mu$. Otherwise, if $i$ belongs uniquely to the line $\ell$, we move $i$ to the closest point of $\mu$.
            \end{itemize}
            
            The total distance that points have moved is a continuous function in $\varepsilon'$.

                \begin{figure}[h]
                \centering
                \includegraphics[scale=0.3]{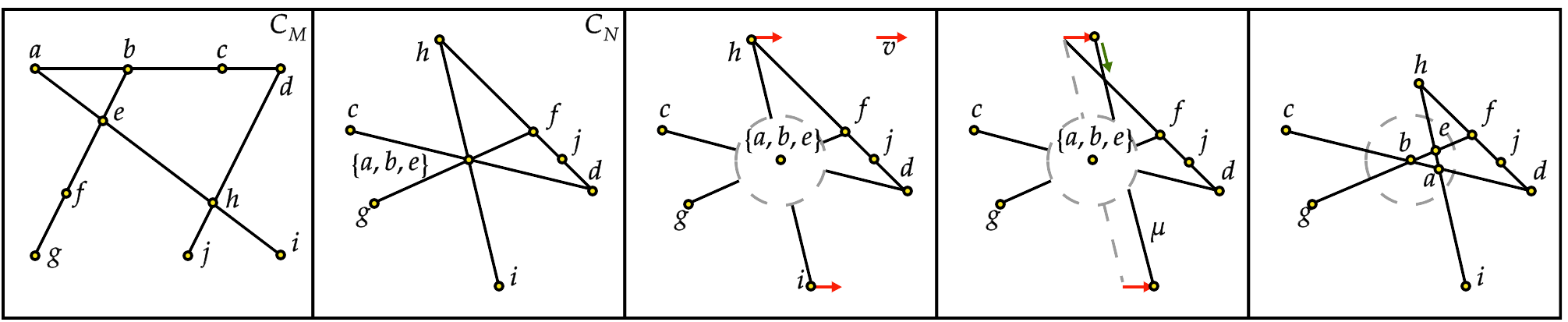}
                \caption{From left to right, the figure shows how a triple point can be solved via the line translation introduced in step S5. The point $\{a,b,e\}$ in $\CCC_N$ is such that $m_{\{a,b,e\}} = 3$. Therefore, it suffices to perform the translation of one of the three lines incident to it. As in Figure~\ref{fig:fat VIII}, the final configuration realises a matroid which is dependent for $M$.}
                \label{fig:fat VII}
                \end{figure}

            \item[S6.] For any $k$, after having performed steps 4 and 5, all lines incident to $I_k \in \mathscr I$ now intersect in points whose coordinates can be obtained by adding a vector $\varepsilon v$ to the coordinates of $I_k$ ($v$ is supposed to have unit norm). Therefore, we can finally resolve all multiple points by taking the intersection of the corresponding lines, as in the right-most representations of Figures~\ref{fig:fat VI} and~\ref{fig:fat VII}.
        \end{itemize}
By construction, the point $C'$ and its associated matroid $N'$ has the desired properties.
\end{proof}   

Before stating our decomposition theorem for quasi-liftable configurations, we state the following remark which allows us to restrict to connected point-line configurations.

    \begin{remark}
    Let $I$ and $J$ be ideals in polynomial rings $\CC[x_1, \dots, x_n]$ and $\CC[y_1, \dots, y_m]$, respectively, with prime decompositions $I = I_1 \cap \dots \cap I_k$ and $J = J_1 \cap \dots \cap J_h$. Then,
$$I + J =
I \otimes \CC[y_1, \dots, y_m] + \CC[x_1, \dots, x_n] \otimes J =
\bigcap_{\substack{i = 1, \dots, k \\ j = 1, \dots, h}} I_i + J_j.$$
Moreover, \cite[Theorem 7.4.$i$]{Mat} ensures that, in our setting, tensor product and finite intersections commute, and \cite[Proposition 5.17.b]{Mil} implies that the ideals $I_i + J_j$ are prime.\hfill $\diamond$
\end{remark}

The above remark applies in particular to disconnected configurations, allowing us not to lose any generality by adding the connectivity assumption to the statement below.
\begin{corollary}[Decomposition theorem for quasi-liftable configurations]\label{decomposition quasi-liftable}
    Let $M$ be a matroid, whose point-line configuration $\CCC_M$ is connected, quasi-liftable and has the property that every point lies on at most two lines. Then,
    \begin{equation}\label{decomposition}
        V_{\C(M)} = V_0 \cup V_{M} \qquad  \hbox{ and equivalently } \qquad \sqrt{I_{\C(M)}} = I_0 \cap I_{M},
    \end{equation}
    where $V_0$ is the matroid variety whose associated configuration is a line with $n$ marked points. Furthermore, the decompositions in (\ref{decomposition}) are, respectively, irreducible and prime.
\end{corollary}

\begin{proof}
  Let $C \in V_{\C(M)}$. By Proposition~\ref{ideal matroid variety}, we have that $C\in V_M^{\comb}$. If $C \in \Gamma_M$, then $C \in V_M = \overline{\Gamma_{M}}$. Thus, from now on, we assume that $C \in V_M^{\comb} \setminus \Gamma_M$ and we prove that $C \in V_0 \cup V_{M}$.
    If $M$ is a matroid over $[n]$, the point $C \in \CC^{3n}$ realises a matroid $N \ge M$ over $[n]$, which is dependent for $M$. 
    
    Now, if the point-line configuration $\CCC_N$ consists of collinear points, then $C \in V_0$. If $\CCC_N$ has more than one line, then $\CCC_M$ has at least two lines as well. In particular, if some sub-configurations of $\CCC_M$ are flattened in $\CCC_N$, then these sub-configurations are liftable as $\CCC_M$ is quasi-liftable and $\CCC_N$ consists of more than one line. Thus, in view of Lemma~\ref{non-simple}, we can assume that $N$ is simple.
    For any line $\ell \in \LL_M$, there is a well-defined projection $\pi_{\ell} \colon \CC^{3n} \to \CC^{3m}$, for some $m \leq n$, which deletes the coordinates of the points which are in $\CCC_M$, but not in $\CCC_M \setminus \ell$.
    Now, for any line $\ell$, let us denote  $M_{\ell}$ for the simple matroid whose point-line configuration is $\CCC_M \setminus \ell$. Then, $\pi_{\ell}(C) \in V_{M_{\ell}}$ because $\CCC_M$ is quasi-liftable.
    Thus,  $C \in \bigcap_{\ell \in \LL} \pi_{\ell}^{-1}(V_{M_{\ell}})$ which is contained in $V_M$ by Theorem~\ref{circuit variety of liftable conf.}.
\end{proof}

\subsection{The ideal of quasi-liftable configurations}
We apply the results of to compute the generators of an ideal whose radical contains the ideal of the matroid variety, see Proposition~\ref{radical}. We will first set up our notation.
  
    \medskip
    
    Let $M$ be a matroid over $[n]$ with point-line configuration $\CCC_M$ and the collinearity matrix $\Lambda$. Consider the notation of Construction~\ref{collinearity matrix}. Then, the non-zero entries of $\Lambda$ are:
    $$ x_i - x_j = \left[ \begin{matrix} x_i & x_j \\
        1 & 1 \end{matrix} \right] = \left[ \begin{matrix} x_i & x_j & 0 \\
        1 & 1 & 0 \\
        0 & 0 & 1
    \end{matrix} \right],$$
    for certain $1 \leq i < j \leq n$, where all columns are the coordinates of points $P_i, P_j$, and $P$, in a frame of reference $\{R_1, R_2, P, U\}$ with $R_1, R_2, U$ being in general position with $P$.
    
    Let $\mathfrak{P}(\Lambda_{i,j})$ be a polynomial in the entries of $\Lambda$, which is the determinant of a square submatrix of $\Lambda$ itself. In this case, $\mathfrak{P}$ is the sum of products of $k$ non-zero entries of $\Lambda$.
    
    In Construction~\ref{collinearity matrix}, the choice of $P$ as a reference point of the projective plane was intended to simplify the notation. In general, one can fix an arbitrary frame of reference $\{R_1, R_2, R_3, U\}$ in $\PP^2_{\CC}$, and use it to take coordinates for $P = (x_P, y_P, z_P)$. Let $\Lambda'$ be the matrix constructed as $\Lambda$, but considering an arbitrary frame of reference. The non-zero entries of $\Lambda'$ are:
    $$ [P_i P_j P] = \left[ \begin{matrix} x_i & x_j & x_P \\
        y_i & y_j & y_P \\
        z_i & z_j & z_P
    \end{matrix} \right].$$
    Let $\mathfrak{P}'(\Lambda'_{i,j})$ be the polynomial defined as $\mathfrak{P}$, but with entries in $\Lambda'$. Now, the vanishing of the polynomial $\mathfrak{P}$ is a projective invariant property. Since change of coordinates are projective transformations, $\mathfrak{P}$ vanishes if and only if the polynomial $\mathfrak{P}'$ vanishes.
    
     By multilinearity of determinants, the polynomial $\mathfrak{P}'$ is the linear combination of $3^k$ copies of $\mathfrak{P}'$ itself, each of them corresponding to a $k$-tuple of points of the frame of reference (with repetitions). Each copy is the sum of products of $k$ $3 \times 3$ determinants of matrices whose first two columns are as in $\mathfrak{P}$ and whose last column is a point of the frame of reference, determined by the corresponding $k$-tuple (see Example~\ref{exa:extensions}). We call each copy of $\mathfrak{P}'$ an \textit{extension} of $\mathfrak{P}.$ \hfill $\diamond$

\begin{example}\label{exa:extensions}
    Consider a matroid $M$ over $[5]$ with circuits $\mathcal{C}(M) = {123, 345}$. The configuration $\mathcal{C}_M$ consists of two lines incident at a point. Take 5 distinct points of $\mathbb{P}^2_{\mathbb{C}}$ lying on the same line and a point $P$ not collinear with them. With the above notation and Construction~\ref{collinearity matrix}, we have:
    $$ \Lambda = \left( \begin{matrix} [23] & - [13] & [12] & 0 & 0 \\
    0 & 0 & [45] & -[35] & [34]\end{matrix}\right), \quad \hbox{ and } \quad \Lambda' = \left( \begin{matrix} [23P] & - [13P] & [12P] & 0 & 0 \\
    0 & 0 & [45P] & -[35P] & [34P]\end{matrix}\right).$$
    Take $\mathfrak{P}$ (resp. $\mathfrak{P}'$) to be the minor generated by the second and third column of $\Lambda$ (resp. $\Lambda'$). Then,
    
    \begin{align*}
        \mathfrak{P}  &= -[13][45],\quad\text{and}\\
        - \mathfrak{P'} &= [13P][45P] \stackrel{\scriptscriptstyle{P = x_P R_1 + y_P R_2 + z_P R_3}}{=} [13 \, x_PR_1+y_PR_2+z_PR_3][45 \, x_PR_1+y_PR_2+z_PR_3] \\
        \hfill &= x_P^2[13R_1][45R_1] + x_Py_P [13R_1][45R_2] + x_Pz_P[13R_1][45R_3] + \\ \hfill &\hspace{10pt }+ x_Py_P[13R_2][45R_1] + y_P^2 [13R_2][45R_2] + y_Pz_P[13R_2][45R_3] + \\ \hfill &\hspace{ 10 pt}+ x_Pz_P[13R_3][45R_1] + y_Pz_P[13R_3][45R_2] + z_P^2[13R_3][45R_3].
    \end{align*}
Here, each summand of $\mathfrak{P}'$ is a copy of $\mathfrak{P}$ associated to a pair $(R_i,R_j)$, with $i,j \in \{1,2,3\}$, which is an extension of $\mathfrak P$. \hfill $\diamond$
\end{example}

\begin{proposition}\label{radical}
    Let $M$ be a maximal matroid over $[n]$, whose point-line configuration $\CCC_M$ is quasi-liftable. Then the ideal $I_M$ is contained in the radical of the ideal $I$ generated by:
    \begin{itemize}
        \item the collinearity conditions of $\CCC_M$;
        \item the extensions of the $n-2$ minors of the collinearity matrix $\Lambda_{\CCC_M}$.
    \end{itemize}
\end{proposition}
\begin{proof}
    From Example~\ref{example:non-maximal}, we know that an $n$-tuple of different points in $\CC^{3n}$ realises $M$ if and only if it contains at least 3 non-collinear points. Every point $P$ in $V(I)$ belongs to the interior of the matroid variety $V_M$ since, after a suitable change of coordinates, it satisfies the equality $\Lambda_{\CCC_M} (z_1 \, \dots \, z_n)^t = (0 \, \dots \, 0)^t$ and has (at least) three non-collinear points in the projective plane (we are considering $n-2$ minors). Thus, $V(I) \subset V(M)$ and so $I_M \subset \sqrt{I}$.
\end{proof}

In the following sections, we will prove that for the quadrilateral set and the $3\times 4$ grid, the two ideals in Proposition~\ref{radical} are equal. Furthermore, we will see that, for both examples, the ideal $I$ is actually radical, and we will provide a minimal generating set for $I$.

\section{The quadrilateral set}\label{sec:qua}

We now apply the results of Section~\ref{sec:circ} to the quadrilateral set. Furthermore, we provide a minimal and geometrically meaningful set of generators for the corresponding ideal. Finally, we outline the method for interpreting these generators from the arrangement, highlighting the underlying symmetries.

\begin{definition}[Quadrilateral set]
A \textit{quadrilateral set} in a projective plane $\PP^2_{\kk}$ is the datum of 4 lines and their 6 intersection points.
\end{definition}

By Definition~\ref{abstract point-line configuration}, the linear point-line configuration corresponding to a quadrilateral set consists of 6 points and 4 lines (see Figure~\ref{fig:qs}), denoted by $\CCC_{QS} = \{\PPPP_{QS}, \LL_{QS}, \II_{QS}\}$ where $\PPPP_{QS} = \{P_1, \dots, P_6\}$, $\LL_{QS} = \{\ell_{123},\ell_{156},\ell_{246},\ell_{345}\}$ and $\II_{QS}$ is the incidence relation showed in the picture.

The associated matroid of the configuration $\CCC_{QS}$ is the simple matroid $QS$ over $[6]$ whose set of circuits are $\min \left(\Delta \cup \binom{[6]}{4} \right)$, where $\Delta = \{123, 156, 246, 345\}$.

The matroid $QS$ is realisable over $\CC$, thus $\CCC_{QS}$ can be embedded in $\PP^2_{\CC}$. We take $\left( x \, y \, z\right)$-coordinates for $1, \dots, 6$, with respect to a fixed frame of reference $\{R_1, R_2, R_3, U\}$. Without loss of generality, we assume that the points of the frame are in general position with any couple of points of the embedded configuration, and we encode the coordinates in the matrix:
\[X = \left(\begin{matrix} x_1 & x_2 & x_3 & x_4 & x_5 & x_6 \\ y_1 & y_2 & y_3 & y_4 & y_5 & y_6 \\ z_1 & z_2 & z_3 & z_4 & z_5 & z_6\end{matrix}\right).\]
Our aim is to understand the variety $V_{QS}$, whose ideal, by Proposition~\ref{ideal matroid variety} is: 
\[
I_{QS} = \sqrt{\langle [123]_X,  [156]_X, [246]_X, [345]_X \rangle \ \colon J_{QS} ^ \infty}
\]
where $J_{QS}$ is the principal ideal $\langle [124]_X \cdot [125]_X \cdots [456]_X \rangle$. We achieve this by identifying a geometrically meaningful set of generators for the ideal $I_{QS}$.

\begin{notation} \label{points in general position}
Let $\ell$ be a line in $\LL$, whose points are $P_{\ell}^1, P_{\ell}^2$ and $P_{\ell}^3$. Let $P, P^1, P^2$ and $P^3$ be three points of the projective plane, not necessarily distinct. Denote as $P_{m}^1, P_{m}^2$ and $P_{m}^3$ the three points of $\PPPP$ which do not belong to $\ell$: two of them will be collinear with $P_{\ell}^1$ (wlog $P_{m}^1$ and $P_{m}^2$), two with $P_{\ell}^2$ (wlog $P_{m}^2$ and $P_{m}^3$) and two with $P_{\ell}^3$ (wlog $P_{m}^1$ and $P_{m}^3$). We denote:
\begin{align*}QS(\ell;P^1,P^2,P^3) &= [P_{\ell}^1P_{m}^1P^1][P_{\ell}^2P_{m}^2P^2][P_{\ell}^3P_{m}^3P^3] - [P_{\ell}^1P_{m}^2P^1][P_{\ell}^2P_{m}^3P^2][P_{\ell}^3P_{m}^1P^3]\\
QS(\ell;P) &= [P_{\ell}^1P_{m}^1P][P_{\ell}^2P_{m}^2P][P_{\ell}^3P_{m}^3P] - [P_{\ell}^1P_{m}^2P][P_{\ell}^2P_{m}^3P][P_{\ell}^3P_{m}^1P].
\end{align*}

Note that, once the line $\ell$ and the points $P, P^1, P^2, P^3$ are fixed, the polynomials above are well-defined up to the sign.  However, for our purposes, we will only concern with their (non)-vanishing.
\end{notation}

\begin{example}
Let us compute the $QS$-polynomials for $\ell =\ell_{123}$, $P^1 = P^2 = R_1$ and $P^3 = R_2$. 
    \begin{align*}QS(\ell_{123};R_1,R_1,R_2) &= [1 5 R_1][2 6 R_1][3 4 R_2] - [1 6 R_1][2 5 R_1][3 4 R_1]\\
    \hfill &= 
    \Big[ \begin{smallmatrix} x_1 & x_5 & 1 \\ y_1 & y_5 & 0 \\ z_1 & z_5 & 0 \end{smallmatrix} \Big]
    \Big[ \begin{smallmatrix} x_2 & x_6 & 1 \\ y_2 & y_6 & 0 \\ z_2 & z_6 & 0 \end{smallmatrix} \Big]
    \Big[ \begin{smallmatrix} x_3 & x_4 & 0 \\ y_3 & y_4 & 1 \\ z_3 & z_4 & 0 \end{smallmatrix} \Big] - \Big[ \begin{smallmatrix} x_1 & x_6 & 1 \\ y_1 & y_6 & 0 \\ z_1 & z_6 & 0 \end{smallmatrix} \Big]
    \Big[ \begin{smallmatrix} x_2 & x_4 & 1 \\ y_2 & y_4 & 0 \\ z_2 & z_4 & 0 \end{smallmatrix} \Big]
    \Big[ \begin{smallmatrix} x_3 & x_5 & 0 \\ y_3 & y_5 & 1 \\ z_3 & z_5 & 0 \end{smallmatrix} \Big]\\
    \hfill &= -x_5y_4y_6z_1z_2z_3+x_4y_5y_6z_1z_2z_3-x_3y_5y_6z_1z_2z_4+x_5y_2y_6z_1z_3z_4+\\
    \hfill & \hspace{12pt} + x_3y_4y_6z_1z_2z_5 - x_4y_1y_6z_2z_3z_5-x_3y_2y_6z_1z_4z_5+x_3y_1y_6z_2z_4z_5 + \\
    \hfill & \hspace{12pt} -x_4y_2y_5z_1z_3z_6+x_5y_1y_4z_2z_3z_6+ x_3y_2y_5z_1z_4z_6-x_5y_1y_2z_3z_4z_6+ \\
    \hfill & \hspace{12pt} -x_3y_1y_4z_2z_5z_6+x_4y_1y_2z_3z_5z_6.
\end{align*}
\end{example}

\begin{remark}
    We introduce two different multi-degrees on the monomials of the ring $R = \CC [x_1, \dots, z_6]$.
    \begin{itemize}
        \item The letter multi-degree $(d_x,d_y,d_z) \in (\ZZ_{\geq 0})^3$, where $d_x$, $d_y$, and $d_z$ are respectively the numbers of $x$,  $y$, and $z$ variables.
        \item The point multi-degree $(d_1, \dots, d_6) \in (\ZZ_{\geq 0})^6$ where $d_i$ is the number of variables corresponding to coordinates of the point $p_i$ for any $i = 1, \dots 6$.
    \end{itemize}
Notice that, by construction, the $QS$-polynomials are homogeneous of point multi-degree $(1, \dots, 1)$. \hfill $\diamond$
\end{remark}

We now provide a family of polynomials that vanish when evaluated on the coordinates of the points of a quadrilateral set.

\begin{theorem}\label{CQS}
Let $\CCC_{QS}$ be a quadrilateral set in $\PP_{\CC}^2$. Then, for any line $\ell \in \LL_{QS}$, and any three points $P^1, P^2, P^3 \in \PP^2_{\CC}$:
$$QS(\ell;P^1,P^2,P^3) = 0.$$
\end{theorem}

\begin{proof}
We have to show that $QS (\ell; P^1, P^2, P^3) = 0$ for any choice of $\ell$ and points $P^1$, $P^2$, $P^3$. As a summand of $QS (\ell; P^1, P^2, P^3)$ vanishes if and only if the other vanishes too, we can assume, without loss of generality, that the points $P^1$, $P^2$, $P^3$ do not belong to any of the lines in $\LL$. Due to the multi-linearity of determinants, the claim follows if $QS(\ell; R_i, R_j, R_k) = 0$, for any line $\ell$ and any $(i, j, k) \in \{1,2,3\}^3$. (Note that we have fixed $\{R_1, R_2, R_3, R_1+R_2+R_3\}$ as the frame of reference).

We will now show the argument for a particular choice of $\ell$. It can be easily adapted for other possible choices. Let us assume $\ell =\ell_{123}$. We want to show that:
\begin{equation} \label{char1}
QS(\ell; R_i, R_j, R_k) = [15R_i][26R_j][34R_k]-[16R_i][24R_j][35R_k] = 0.
\end{equation}
By construction, $6 \in\ell_{156}$. Now, $\{1,5, 1+5\}$ is a frame of reference for the line $\ell_{156}$. This means that there exists a unique choice of $\alpha, \beta \in \CC$ such that $\alpha 1 + \beta 5 = 6$. Here, it is important to remark that $1, \dots, 6$ are fixed representatives of the corresponding points in $\PP^2_{\CC}$, which makes the choice of $\alpha$ and $\beta$ unique. If we plug this into $QS(\ell; R_i, R_j, R_k)$, we obtain:
\begin{align*}
    QS(\ell; R_i, R_j, R_k) &= [15R_i][26R_j][34R_k]-[16R_i][24R_j][35R_k]\\
    \hfill &= [15R_i][26R_j][34R_k]-[1 \, \mathbf{\alpha 1 + \beta 5} \, R_i][24R_j][35R_k]\\
    \hfill &\hspace{-13 pt}\stackrel{\hbox{\footnotesize{multilin.}}}{=}  [15R_i][26R_j][34R_k]-\alpha[11 R_i][24R_j][35R_k] - \beta[1 5 R_i][24R_j][35R_k]\\
    \hfill &= [15R_i][26R_j][34R_k] - \beta[1 5 R_i][24R_j][35R_k].
\end{align*}
In this way, we managed to have the same term as the first factor of both products. Via the same argument, the structure of the quadrilateral set yields also that $4 = \alpha'2 + \beta'6$ and $5 = \alpha''3 + \beta'' 4$, for a unique choice of $\alpha',\alpha'', \beta', \beta'' \in \CC$. Finally, we have that:
\[QS(\ell; R_i, R_j, R_k) = [15R_i][26R_j][34R_k] - \beta\beta'\beta''[1 5 R_i][26R_j][34R_k] = (1 - \beta\beta'\beta'') [15R_i][26R_j][34R_k]. \]
As a consequence, proving (\ref{char1}) is equivalent to show that $\beta \beta' \beta'' =1$. On the other hand,
\begin{align*}
5 &= \alpha''3 + \beta'' 4\\
\hfill &= \alpha''3 + \alpha'\beta'' 2 + \beta'\beta'' 6\\
\hfill &= \alpha''3 + \alpha'\beta'' 2 + \alpha \beta' \beta'' 1 + \beta\beta'\beta'' 5.
\end{align*}
Thus, $(1 - \beta\beta'\beta'') 5 = \alpha''3 + \alpha'\beta'' 2 + \alpha \beta' \beta'' 1$. Here, on the l.h.s., there is another representation of point $5$; whereas, on the r.h.s., there is a point in the line $\ell_{123}$. As far as $5 \notin\ell_{123}$ by construction, the equality above holds if and only if both sides give $\left(\begin{smallmatrix} 0 \\ 0 \\ 0 \end{smallmatrix}\right)$; that is if and only if $\beta \beta'\beta'' = 1$, as desired.
\end{proof}

Furthermore, projective transformations keep track of the vanishing of these polynomials.

\begin{lemma}\label{projective invariant property}
The vanishing of a polynomial $QS (\ell; R_i, R_j, R_k) = 0$ is a projective invariant property. 
\end{lemma}
\begin{proof}
We prove the lemma for the line $\ell_{123}$. We need to show that:
\[QS(\ell_{123}; R_i, R_j, R_k) =[15R_i][26R_j][34R_j]-[16R_i][24R_j][35R_k] =0 \]
is a projective invariant property. Therefore, we consider $T \in GL(\CC,3)$ and $D \in \diag(\CC,3)$, and we write down $QS(T\ell_{123}D;TPD)$, as follows:
\begin{align*}
QS(T\ell_{123}D; TR_iD, TR_jD, TR_kD) &= [T1D\;T5D\;TR_iD][T2D\;T6D\;TR_jD][T3D\;T4D\;TR_kD]+\\
\hfill & \hspace{15pt} -[T1D\;T6D\;TR_iD][T2D\;T4D\;TR_jD][T3D\;T5D\;TR_kD]\\
\hfill &= \det T^3 \det D^3 ([15R_i][26R_j][34R_k]-[16R_i][24R_j][35R_k]).
\end{align*}
At this stage, we can see that
\[QS(T\ell_{123}D; TR_iD, TR_jD, TR_kD) =0 \quad \Longleftrightarrow  \quad QS(\ell_{123}; R_i, R_j, R_k) =0\] which completes the proof. The same argument applies analogously to other choices of lines.
\end{proof}

In \cite{Ric, Stu}, the vanishing of the bracket polynomials $QS(\ell; R_3)$ is proved to characterise the liftability of six points in $\PP^1_{\CC}$ to a quadrilateral set. More generally, the whole family of $QS$-polynomials just introduced offers a characterisation of the liftability of a 6-tuple of collinear points in $\PP^2_{\CC}$ to a quadrilateral set. In particular, we prove the following result.

\begin{theorem}\label{TFAE quad set}
    Let $r$ be a line in $\PP^2_{\CC}$ and $1, \dots, 6$ distinct points of $r$. Consider the collection $\LL=\{\ell_{123},\ell_{156},\ell_{246},\ell_{345}\}$, where $\ell_{ijk}$ is the combinatorial line consisting of points $i,j$, and $k$. Then, the following statements are equivalent:
    \begin{itemize}
        \item[$i$.] The points $1, \dots, 6$ are the projective image of a quadrilateral set.
        \item[$ii$.] The polynomials $QS(\ell; P^1, P^2, P^3)$ vanish for any $\ell \in \LL$ and any $P^1,P^2,P^3 \in \PP^2_{\CC}$.
        \item[$iii$.] The polynomials $QS(\ell; R_i, R_j, R_k)$ vanish for any $\ell \in \LL$ and any $(i,j,k) \in \{1,2,3\}^3$.
        \item[$iv$.] The polynomials $QS(\ell_{123}; R_i, R_j, R_k)$ vanish for any $(i,j,k) \in \{1,2,3\}^3$.
        \item[$v$.] The polynomials $QS(\ell_{123}; R_i, R_j, R_k)$ vanish for any $(i,j,k) \in \{1,2,3\}^3$ with $i \leq j \leq k$.
    \end{itemize}
\end{theorem}

\begin{proof}
It is immediate to see that $(ii) \Longrightarrow (iii) \Longrightarrow (iv) \Longrightarrow (v)$.\\

\textbf{$(i) \Longrightarrow (ii)$} By Lemma~\ref{projective invariant property}, the vanishing of the $QS$ polynomials is a projective invariant property, thus when six collinear points are the projective image of a quadrilateral set, $QS(\ell; P^1, P^2, P^3) = 0$ where $\ell, P^1, P^2$, and $P^3$ satisfy the assumptions of Theorem~\ref{CQS}.

\smallskip

\textbf{$(ii) \Longrightarrow (i)$} Conversely, let $1,\ldots,6$ be collinear points that make the polynomials $QS(\ell; P^1,P^2;P^3)$ vanish for $\ell \in \{\ell_{123},\ell_{156},\ell_{246},\ell_{345}\}$.
We then computationally verify that the vanishing of the $QS$ polynomials is equivalent to requiring the collinearity matrix $\Lambda_{QS}$ not to have maximal rank.
$$ \Lambda_{QS} = \left( \begin{matrix}
[23] & -[13] & [12] & 0 & 0 & 0 \\
[56] & 0 & 0 & 0 & -[16] & [15] \\
0 & [46] & 0 & -[26] & 0 & [24] \\
0 & 0 & [45] & -[35] & [34] & 0    
\end{matrix} \right)$$
As a consequence, the linear system $\Lambda_{QS} (z_1 \dots z_6)^t = (0 \dots 0)^t$ has a solution space of dimension at least 3. In other words, by Construction \ref{collinearity matrix} there exists a non-degenerate quadrilateral set whose image via the projection through $P$ on the line $r$ consists of exactly points $1, \dots, 6$.

\smallskip

\textbf{$(iii) \Longrightarrow (ii)$} The implication is followed by the following equalities of ideals in $\CC[x_1, \dots, z_6]$:
\begin{multline*}
\langle [123], [156], [246], [345], QS(\ell; P^1, P^2, P^3) \hbox{ for } \ell \in \LL \hbox{ and } P^1,P^2,P^3 \in \PP^2_{\CC} \rangle = \\ =\langle [123], [156], [246], [345], QS(\ell; R_i, R_j, R_k) \hbox{ for } \ell \in \LL \hbox{ and } (i,j,k) \in \{1,2,3\}^3\rangle.\end{multline*}
 However, the equality of the two ideals holds by multi-linearity of determinants.

 \smallskip

\textbf{$(iv) \Longrightarrow (iii)$} Similarly to the previous implication, we check that:
\begin{multline*}
\langle [123], [156], [246], [345], QS(\ell; R_i, R_j, R_k) \hbox{ for } \ell \in \LL \hbox{ and } (i,j,k) \in \{1,2,3\}^3\rangle = \\ = \langle [123], [156], [246], [345], QS(\ell_{123}; R_i, R_j, R_k) \hbox{ for } (i,j,k) \in \{1,2,3\}^3\rangle.\end{multline*} More generally, if we fix a line $\ell$, the polynomials $QS(\ell; R_i, R_j, R_k)$ together with $[123]_X,  [156]_X, [246]_X,$ and $[345]_X$ generate the corresponding polynomials for the other lines.

Assuming that $QS(\ell_{123}; R_i, R_j, R_k) =0$ for any $(i,j,k) \in \{1,2,3\}^3$ for any line $\ell \in \LL \setminus \{\ell_{123}\}$, we construct a projection $\phi$ that sends $\ell_{123}$ to $\ell$ and keeps the quadrilateral set globally fixed (so that the intersection points in $\PPPP$ are just permuted by $\phi$).
We can define these projections explicitly. The points $\{1,2,5,4\}$ are a frame of reference for the projective plane and each permutation of these four points defines uniquely a projection $\phi$ that permutes the 3 diagonal points ($3,6, 14 \wedge 25$). If we add the requirement that $14 \wedge 25$ is fixed (which is necessary for the $\phi$-stability of the quadrilateral set), we find the three desired projections.
By Lemma~\ref{projective invariant property} the vanishing of $QS(\ell_{123}; R_i, R_j, R_k) =0$ is a projective invariant property. Hence, to obtain all the polynomials $QS (\ell; R_i, R_j, R_k)$ we can choose the generators among the polynomials $QS(\ell_{123}, R_i, R_j, R_k)$.

\smallskip

\textbf{$(v) \Longrightarrow (iv)$} We check the following equality of ideals in $\CC[x_1, \dots, z_6]$:
\begin{multline*}
\langle [123], [156], [246], [345], QS(\ell_{123}; R_i, R_j, R_k) \hbox{ for } (i,j,k) \in \{1,2,3\}^3\rangle = \\ = \langle [123], [156], [246], [345], QS(\ell_{123}; R_i, R_j, R_k) \hbox{ for } (i,j,k) \in \{1,2,3\}^3 \hbox{ with } i \leq j \leq k \rangle.
\end{multline*}
We point out that a choice of $(i,j,k)$ determines the letters appearing in the terms of $QS(\ell_{123}; R_i, R_j, R_k)$. For example, $QS(\ell_{123}; R_1, R_1, R_2)$ is the sum of monomials of degree 6, where each of them are the product of one $x$ coordinate, two $y$ coordinates, and three $z$ coordinates. The same holds for $QS(\ell_{123}; R_1, R_2, R_1)$ and $QS(\ell_{123}; R_2, R_1, R_1)$. In spite of that, $$QS(\ell_{123}; R_1, R_1, R_2) \neq QS(\ell_{123}; R_1, R_2, R_1) \neq QS(\ell_{123}; R_2, R_1, R_1).$$ 
In other words, the selection of $(i,j,k)$ determines the letter multi-degree of the $QS$ polynomial. However, multiple choices of $(i,j,k)$ can yield the same letter multi-degree. Specifically, having a generator per letter multi-degree is sufficient to generate all the polynomials. The remaining $QS$-polynomials with the same multi-degree are obtained by adding the corresponding generator to a polynomial combination of $[123]$, $[156]$, $[246]$, and $[345]$, with coefficients stored in Table~\ref{tab:coefficients}.
\end{proof}

    \begin{table}[h!]
        \centering
        \begin{tabular}{| c | c || c | c | c | c|}
         \hline
         $i,j,k$ & Gener. & Coeff. of $[123]$ & Coeff. of $[156]$ & Coeff. of $[246]$ & Coeff. of $[345]$ \\
         \hline \hline  $ 1,2,1 $ & $ 1,1,2 $ & $ -y_6z_4z_5+y_5z_4z_6 $ & $ y_3z_2z_4-y_2z_3z_4 $ & $ -y_5z_1z_3+y_1z_3z_5 $ & $  -y_6z_1z_2+y_1z_2z_6 $\\
         \hline         $ 2,1,1 $ & $ 1,1,2 $ & $ -y_6z_4z_5+y_4z_5z_6 $ & $ y_4z_2z_3-y_2z_3z_4 $ & $ -y_3z_1z_5+y_1z_3z_5 $ & $  -y_6z_1z_2+y_2z_1z_6 $\\
         \hline  \hline $ 1,3,1 $ & $ 1,1,3 $ & $ y_4y_6z_5-y_4y_5z_6 $ & $ -y_3y_4z_2+y_2y_4z_3 $ & $ y_3y_5z_1-y_1y_3z_5 $ & $ y_2y_6z_1-y_1y_2z_6 $\\
         \hline         $ 3,1,1 $ & $ 1,1,3 $ & $ y_5y_6z_4-y_4y_5z_6 $ & $ -y_3y_4z_2+y_2y_3z_4 $ & $ y_3y_5z_1-y_1y_5z_3 $ & $ y_1y_6z_2-y_1y_2z_6$\\
         \hline  \hline $ 2,1,2 $ & $ 1,2,2 $ & $ x_5z_4z_6-x_4z_5z_6 $ & $ -x_4z_2z_3+x_3z_2z_4 $ & $ -x_5z_1z_3+x_3z_1z_5 $ & $ -x_2z_1z_6+x_1z_2z_6  $\\
         \hline         $ 2,2,1 $ & $ 1,2,2 $ & $ x_6z_4z_5-x_4z_5z_6 $ & $ -x_4z_2z_3+x_2z_3z_4 $ & $ x_3z_1z_5-x_1z_3z_5 $ & $ x_6z_1z_2-x_2z_1z_6  $\\
         \hline  \hline $ 1,3,2 $ & $ 1,2,3 $ & $ -x_4y_6z_5+x_4y_5z_6 $ & $ x_4y_3z_2-x_4y_2z_3 $ & $ -x_3y_5z_1+x_3y_1z_5 $ & $ -x_2y_6z_1+x_2y_1z_6 $\\
         \hline         $ 2,1,3 $ & $ 1,2,3 $ & $ -x_5y_4z_6+x_4y_5z_6 $ & $ x_4y_3z_2-x_3y_4z_2 $ & $ x_5y_3z_1-x_3y_5z_1 $ & $ x_2y_1z_6-x_1y_2z_6 $\\
         \hline         $ 2,3,1 $ & $ 1,2,3 $ & $ -x_6y_4z_5+x_4y_5z_6 $ & $ x_4y_3z_2-x_2y_4z_3 $ & $ -x_3y_5z_1+x_1y_3z_5 $ & $ -x_6y_2z_1+x_2y_1z_6 $\\
         \hline         $ 3,1,2 $ & $ 1,2,3 $ & $ -x_5y_6z_4+x_4y_5z_6 $ & $ x_4y_3z_2-x_3y_2z_4 $ & $ -x_3y_5z_1+x_5y_1z_3 $ & $ -x_1y_6z_2+x_2y_1z_6 $\\
         \hline         $ 3,2,1 $ & $ 1,2,3 $ & $ -x_6y_5z_4+x_4y_5z_6 $ & $ x_4y_3z_2-x_2y_3z_4 $ & $ -x_3y_5z_1+x_1y_5z_3 $ & $ -x_6y_1z_2+x_2y_1z_6 $\\
         \hline \hline  $ 3,1,3 $ & $ 1,3,3 $ & $ x_5y_4y_6-x_4y_5y_6 $ & $ -x_4y_2y_3+x_3y_2y_4 $ & $ -x_5y_1y_3+x_3y_1y_5 $ & $ -x_2y_1y_6+x_1y_2y_6 $\\
         \hline         $ 3,3,1 $ & $ 1,3,3 $ & $ x_6y_4y_5-x_4y_5y_6 $ & $ -x_4y_2y_3+x_2y_3y_4 $ & $ x_3y_1y_5-x_1y_3y_5 $ & $ x_6y_1y_2-x_2y_1y_6 $\\
         \hline \hline  $ 2,3,3 $ & $ 2,2,3 $ & $ x_4x_6z_5-x_4x_5z_6 $ & $ -x_3x_4z_2+x_2x_4z_3 $ & $ x_3x_5z_1-x_1x_3z_5 $ & $ x_2x_6z_1-x_1x_2z_6 $\\
         \hline         $ 3,2,2 $ & $ 2,2,3 $ & $ x_5x_6z_4-x_4x_5z_6 $ & $ -x_3x_4z_2+x_2x_3z_4 $ & $ x_3x_5z_1-x_1x_5z_3 $ & $ x_1x_6z_2-x_1x_2z_6  $\\
         \hline \hline  $ 3,2,3 $ & $ 2,3,3 $ & $ -x_5x_6y_4+x_4x_6y_5 $ & $ x_2x_4y_3-x_2x_3y_4 $ & $ x_1x_5y_3-x_1x_3y_5 $ & $ x_2x_6y_1-x_1x_6y_2 $\\
         \hline         $ 3,3,2 $ & $ 2,3,3 $ & $ -x_5x_6y_4+x_4x_5y_6 $ & $ x_3x_4y_2-x_2x_3y_4 $ & $ -x_3x_5y_1+x_1x_5y_3 $ & $ -x_1x_6y_2+x_1x_2y_6  $\\
         \hline
         \end{tabular}
        \caption{This table allows to reconstruct how the polynomials which have been removed from item $(iii)$ to item $(iv)$ can be written as a polynomial combination of the generators in $(iv)$. Each block of rows corresponds to a different letter multi-degree. In the first column, we select the excluded choices of $R_i, R_j, R_k$, and, in the second one, there is the generator of the ideal in $(iv)$ having the same letter multi-degree. Each of the excluded polynomials is the sum of the corresponding generator and a combination of $[123], [156], [246],$ and $[345]$ whose coefficients are specified in the corresponding entry of the table.}
        \label{tab:coefficients}
    \end{table}

Applying the characterisation in Theorem~\ref{TFAE quad set}, we now compute a minimal generating set for $I_{QS}$. 

\begin{theorem}\label{Generators}
The associated ideal $I_{QS}$ of the  quadrilateral set is minimally generated as: 
\begin{equation}\label{generators}
I_{QS} =  \langle [123]_X,  [156]_X, [246]_X, [345]_X , QS(\ell_{123}; R_i, R_j, R_k) \, \forall \, i \leq j\leq k, \hbox{ with } i,j,k \in \{1,2,3\}\rangle.\end{equation}
\end{theorem}

\begin{proof}
    For the ease of notation, we denote $I$ for the ideal in the right hand side of \eqref{generators}.
   By Theorem~\ref{CQS}, we have that $I \subset I_{QS}$.
By numerical computations available on GitHub,  \begin{center}        \url{https://github.com/ollieclarke8787/PointAndLineConfigurations}.
        \end{center}
 we verify that the natural GRevLex term order gives a square-free initial ideal for $I$.  Thus, the ideal $I$ is radical, and verifying that $I \supset I_{QS}$ is equivalent to show that $V(I) \subset V(I_{QS})$. 
 
 Let $A \in \CC^{18}$ be a point in $V(I)$. The coordinates $A_1^1, A_1^2, A_1^3, \dots, A_6^1, A_6^2, A_6^3$ of $A$ can be seen as the $(x,y,z)$-coordinates of 6 points in the projective plane, that can be represented by the $3 \times 6$ matrix:
    \[ A = \left( \begin{matrix}
        A_1^1 & A_2^1 & \dots & A_6^1 \\
        A_1^2 & A_2^2 & \dots & A_6^2 \\
        A_1^3 & A_2^3 & \dots & A_6^3
    \end{matrix} \right)\]
    The columns of $A$ generate a realisable matroid $M_A$ which corresponds to a point-line configuration $\CCC_A$ (as the columns of $A$ are coordinates of points).
  Since $A \in V(I)$, $A$ satisfies the determinantal collinearity conditions. Hence, $A$ is a point in the combinatorial closure of the matroid $QS$. Corollary~\ref{decomposition quasi-liftable} ensures that after a suitable perturbation, $A$ becomes either a realisation of the matroid corresponding to a 6-pointed line or a realisation of $QS$. 
  
    The fact that $A \in V(I)$ implies that the coordinates of $A$ satisfy also the $QS$ polynomials, implying that, if the points of $A$ lie on a line, then $A$ is the projective image of a quadrilateral set and therefore becomes the realisation of a quad-set with an arbitrary small lifting. This holds also if only some of the points in $C_A$ are loops. To check this is enough to consider the case where we have 5 points on a line and a loop. We want to verify the existence of a sixth point in the line such that $1, \dots, 6$ is the projective image of a quadrilateral set. To this purpose, following Figure~\ref{fig:quad-set lifting}, we consider the two lines of the quadrilateral-set which are not involved by the loop point and represent them as collinear points in the fibres of the non-loops. At this stage, the sixth point of the quadrilateral set is uniquely determined by the intersection of the corresponding ghost lines, and its projection on $\ell$ is the point we wanted. 
Therefore, $A$ is in the Euclidean closure of $\Gamma_{QS}$. Hence, $A \in V(I_{QS})$.

    \begin{figure}[h]
        \centering
        \includegraphics[scale=0.3]{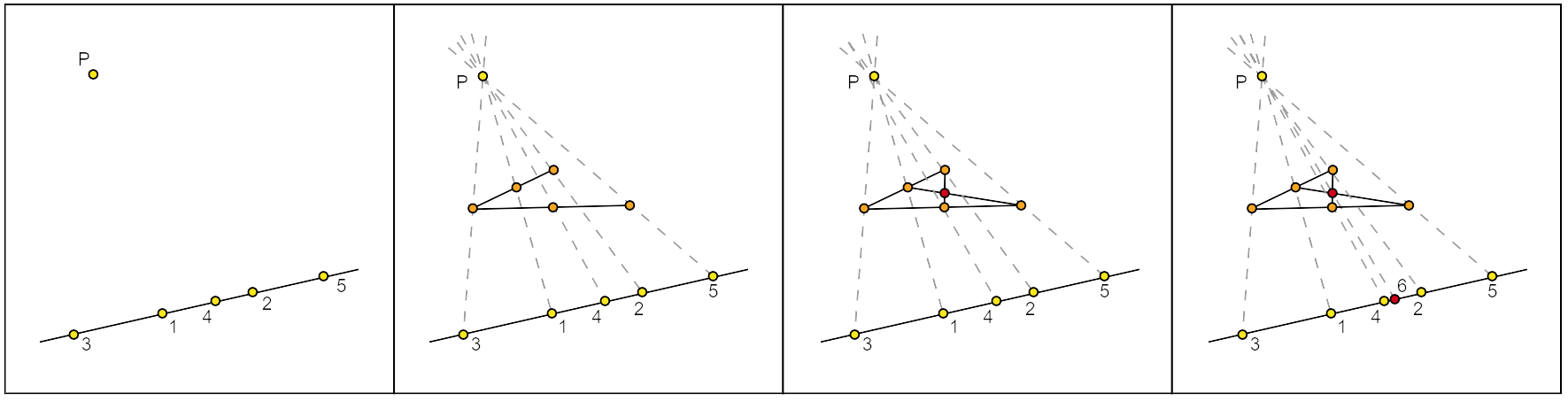}
        \caption{From left to right, this figure justify the existence of at least a choice of 6 completing $1, \dots, 5$ to a projective image of a quadrilateral set.}
        \label{fig:quad-set lifting}
    \end{figure}
 
We now prove the minimality of the generating set. In particular, we show that all the polynomials we used to generate $I$ are independent.
    By reasons of degree and point multi-degree, the generators of degree 3 are polynomially independent over $R$. Hence, it is enough to check that generators of degree 6 having different letter multi-degree are polynomially independent over $\CC[x_1, \dots, z_6]$. The fact that all the generators have point multi-degree $(1,1,1,1,1,1)$ and letter multi-degree fixed by the partition of variables plays a central role in this proof. 

    First, we note that any choice of $(i,j,k) \in {1,2,3}^3$ results in certain variables not appearing in the polynomial $QS(\ell_{123}; R_i, R_j, R_k)$. In Table~\ref{tab:missingvariables} below, each row corresponds to a choice of $(i,j,k)$ with $i \leq j \leq k$ and indicates the variables that do not appear in the corresponding generator.
\begin{table}[h]
    \centering
    \begin{tabular}{ |c||c|c|c|c|c|c| }
         \hline
         $(i,j,k)$ & 1 & 2 & 3 & 4 & 5 & 6 \\
         \hline \hline  $(1,1,1)$ & $x_1$ & $x_2$ & $x_3$ & $x_4$ & $x_5$ & $x_6$ \\
         \hline $(2,2,2)$ & $y_1$ & $y_2$ & $y_3$ & $y_4$ & $y_5$ & $y_6$ \\
         \hline $(3,3,3)$ & $z_1$ & $z_2$ & $z_3$ & $z_4$ & $z_5$ & $z_6$ \\
         \hline \hline  $(1,1,2)$ & $x_1$ & $x_2$ & $y_3$ & - & - & $x_6$ \\
         \hline  $(1,1,3)$ & $x_1$ & $x_2$ & $z_3$ & - & - & $x_6$ \\
         \hline  $(1,2,2)$ & $x_1$ & $y_2$ & $y_3$ & $y_4$ & - & - \\
         \hline  $(1,3,3)$ & $x_1$ & $z_2$ & $z_3$ & $z_4$ & - & - \\
         \hline  $(2,2,3)$ & $y_1$ & $y_2$ & $z_3$ & - & - & $y_6$ \\
         \hline $(2,3,3)$ & $y_1$ & $z_2$ & $z_3$ & $z_4$ & - & - \\
         \hline \hline  $(1,2,3)$ & $x_1$ & $y_2$ & $z_3$ & - & - & - \\
         \hline
    \end{tabular}
    \caption{For each generator of degree 6, identified by a choice of $i,j,$ and $k$ in the first column, the table shows which variables do not appear in the polynomial.}
    \label{tab:missingvariables}
\end{table}

Keeping this in mind, we can deduce that the generators of letter multi-degree $(3,3,0)$, $(3,0,3)$, $(0,3,3)$, namely, the ones corresponding to the first three rows in Table~\ref{tab:missingvariables}, are polynomially independent with respect to the others. For ease of computation, let us focus on the specific case of $QS(\ell_{123}; R_1, R_1, R_1)$, which has a letter multi-degree of $(0,3,3)$ and does not contain any $x$ variables. Since all the other generators are the sum of monomials containing $x$ variables, any polynomial combination of them would result in the sum of monomials with a multi-degree of $(a,b,c)$ where $a > 0$. This demonstrates that $QS(\ell_{123}; R_1, R_1, R_1)$ is not contained in the ideal generated by the other generators. An analogous argument applies to the generators with letter multi-degrees of $(3,0,3)$ and $(3,3,0)$.

\smallskip

Now, let us consider the other generators of degree 6. We show preliminarily the following claim.

\begin{claim} In the notation above, none of the generators having a permutation of $(1,2,3)$ as letter multi-degree, i.e. rows 4 to 9 in Table~\ref{tab:missingvariables}, is contained in $\langle [123]_X,  [156]_X, [246]_X, [345]_X \rangle$.\end{claim}

\begin{proof}[{\bf Proof of the claim}]
If $QS(\ell_{123}; R_1, R_1, R_2)$ were in the ideal $\langle [123]_X, [156]_X, [246]_X, [345]_X \rangle$, it would be a polynomial combination of the four generators with at least one non-zero coefficient of letter multi-degree $(0,1,2)$. Three $x$ variables never appear; namely $x_1, x_2$, and $x_6$: these appear in $[123]_X, [156]_X, [246]_X$. If any of these variables appear in a monomial of the combination, this monomial must be canceled out using the other degree 3 generators containing the same variable.

Now, assume by contradiction that $[123]_X$ is multiplied by a non-zero coefficient $y_lz_mz_n$ of point multi-degree $(0,0,0,1,1,1)$. Here $l \neq 4,5$, for reasons of point multi-degree. So the only possibility is $y_6z_4z_5$. Then, two $x_1$ monomials must be cancelled: $x_1y_2y_6z_3z_4z_5 - x_1y_3y_6z_2z_4z_5$. For this purpose, we have to multiply $[156]_X$ by $y_2z_3z_4+y_3z_2z_4$. The second term gives rise to an $x_6$ monomial which cannot be cancelled by reasons of point multi-degree. A symmetric argument works starting from $[156]_X, [246]_X$. Thus, the coefficients of $[123]_X,  [156]_X, [246]_X$ have to be zero. 
The only possibility left is that $QS(\ell_{123};R_1, R_1, R_2)$ is the product of $[345]_X$ by a homogeneous polynomial of degree 3, letter multi-degree $(0,1,2)$ and point multi-degree $(1,1,0,0,0,1)$. But this can be excluded because such a multiplication cannot give rise to any of the monomials $x_ay_by_cz_4z_5z_d$ appearing in $QS(\ell_{123}; R_1, R_1, R_2)$.

By symmetry, an equivalent strategy can be performed with the generators of letter multi-degree $(1,1,3), (1,2,2), (1,3,3), (2,2,3)$ and $(2,3,3)$.
\end{proof}
Now, by contrary assume that a 6-degree generator is a polynomial combination of the others, i.e.,
\[g = \sum_{i = 1}^4 p_i g_{3,i} + \sum_{i = 1}^9 k_i g_{6,i}\] where $p_i \in R$, $k_i \in \kk$, the $g_{3,i}$'s are the generators of degree 3 and the $g_{6,i}$'s are the generators of degree 6. By reason of letter multi-degree, none of the monomials of $g$ can arise from the generators of 6-th degree. In addition, the product of a monomial for a degree 3 generator either gives rise to monomials of the same letter multi-degree of $g$ or to monomials of the same letter multi-degree of one of the other generators. This implies, in turn, that both $g$ and $\sum_{i = 1}^9 k_i g_{6,i}$ are in $\langle [123]_X,  [156]_X, [246]_X, [345]_X\rangle$. 

If $g$ has a letter multi-degree of $(1,2,3)$ (possibly permuted), the statement follows because it leads to a contradiction with the claim. If $g$ is the generator with a letter multi-degree of $(2,2,2)$, then we can conclude due to the fact that $\sum_{i = 1}^9 k_i g_{6,i} \in \langle [123]_X, [156]_X, [246]X, [345]X\rangle$. Indeed, since the $g{6,i}$'s all have different letter multi-degrees, this implies that at least one of the $g{6,i}$'s is contained in $\langle [123]_X, [156]_X, [246]_X, [345]_X\rangle$, contradicting the claim.
\end{proof}

As an immediate consequence  of Corollary~\ref{decomposition} and Theorem~\ref{Generators}, we have the following:
\begin{corollary}
    The circuit variety $V_{\C (QS)}$ decomposes irreducibly as:
\[
 V_{\C(QS)} = V_0 \cup V_{QS}= V(\langle [i j k] \, | \, \{i,j,k\} \subseteq \{1, \dots, 6\} \rangle) \cup V(I_{QS}).
\]
\end{corollary}

\section{The $3 \times 4$ grid matroid}\label{sec:gri}
In this section, we focus on the $3 \times 4$ grid matroid and its defining equations. This example is chosen based on the work of \cite{PfSt}, where a generating set for the matroid ideal is computed using a specialised algorithm tailored for this configuration. Here, we give a geometric description of the generators. 

\begin{definition}[$n \times m$ grid configuration]
    An $n \times m$ \textit{grid configuration} is a plane point-line configuration consisting of the $n \cdot m$ points of intersection of two pencils of $n$ and $m$ parallel lines.
\end{definition}

    When we take the realisation of a grid, the realisations of two lines of the same pencil will intersect in the projective plane. However, these intersection points will not be part of the realisation itself. 

\begin{example}
    Our main example is the $3\times 4$ grid $G^3_4$. Mirroring our techniques for the quadrilateral set, we begin with the linear point-line configuration with 7 lines and 12 points $\CCC_{G} = \left( \PPPP_{G^3_4}, \LL_{G^3_4}, \II_{G^3_4} \right)$, see Figure~\ref{fig:3x4g}. Following the diagram, the lines are referred to as \textit{rows} and \textit{columns} in $\LL_{G^3_4}$. The underlying simple matroid $G^3_4$ on $[12]$, has circuits $\min \left( \Delta \cup \binom{[12]}{4} \right)$ where $\Delta = \{123, \dots$ $ 10 \, 11 \, 12\}$.\hfill $\diamond$
\end{example}

The matroid $G^3_4$ is realisable over $\CC$ and any realisation is represented as a $3 \times 12$ matrix whose columns are coordinates of points in $\PP^2_{\CC}$, with respect to a reference $\{R_1, R_2, R_3, U\}$. In view of Proposition~\ref{ideal matroid variety}, we study the ideal $I_{G^3_4} = \sqrt{I_{\C\left(G^3_4\right)} \ \colon J_{G^3_4} ^ \infty}$ of the algebraic variety $V_{G^3_4}$.

\begin{notation} 
\begin{itemize}
    \item 
Let $c_{i}$ be a column in $\LL_{G^3_4}$, whose points are $P^1_i, P^2_i$ and $P^3_i$, where the upper index labels the row. Then, the points in $G^3_4 \setminus c_i$ are in natural bijection with a 3 $\times$ 3 matrix.

\item Let $\Sigma$ be the set of permutation matrices $\sigma$ over three elements. For $j = 1,2,3$, we denote as $\sigma^j$ the non-zero entry of the $j-$th row and as $P_{\sigma}^j$ the corresponding point in $G^3_4 \setminus c_i$.

\item For $k \in \{1,2,3,4\} \setminus \{i\} = \{k_1, k_2, k_3\}$, each column $c_k$ has a single point paired with a non-zero entry of $\sigma$. We label $Q^{\sigma,1}_k$ and $Q^{\sigma,2}_k$ the points in $\PPPP_{G^3_4}$ which belong to $c_k \setminus \{P_{\sigma}^j\}$. In particular, $Q^{\sigma,1}_k$ is the point with the lowest index.

\item Let $ P^1,\dots, P^6$ be six points in $\PP^2_{\CC}$, not necessarily distinct. We introduce the polynomials:
\begin{align*}
G^3_4(c_i;P^1,\dots,P^6) &= \sum _{\sigma \in \Sigma}\sgn (\sigma)[P^1_iP_{\sigma}^1P^1][P^2_iP_{\sigma}^2P^2][P^3_iP_{\sigma}^3P^3][Q_{k_1}^{\sigma,1} Q_{k_1}^{\sigma,2}P^4][Q_{k_2}^{\sigma,1} Q_{k_2}^{\sigma,2}P^5][Q_{k_3}^{\sigma,1} Q_{k_3}^{\sigma,2}P^6],\\
G^3_4(c_i;P) &= \sum _{\sigma \in \Sigma}\sgn (\sigma)[P^1_iP_{\sigma}^1P][P^2_iP_{\sigma}^2P][P^3_iP_{\sigma}^3P][Q_{k_1}^{\sigma,1} Q_{k_1}^{\sigma,2}P][Q_{k_2}^{\sigma,1} Q_{k_2}^{\sigma,2}P][Q_{k_3}^{\sigma,1} Q_{k_3}^{\sigma,2}P].
\end{align*}
\end{itemize}
\end{notation}

   The $G^3_4$ polynomials are defined by a sum over all permutation matrices, making them independent of the bijection between the points of $G^3_4 \setminus c_i$ and the entries of a $3 \times 3$ matrix.~This choice only affects the sign of the polynomials.~We~focus~on~configurations that cause the $G^3_4$
   ~polynomials~to~vanish.

\begin{theorem}\label{CG34}
Let $\CCC$ be a $ 3 \times 4$ grid configuration in $\PP_{\CC}^2$. Then, for any column $c_i$ and any six points $P^1,\dots$,$P^6$ in general position w.r.t.~to any couple of points of the configuration, $G^3_4(c_i;P^1,\dots,P^6) = 0$.
\end{theorem}

\begin{proof}
Due to the multi-linearity of determinants, the statement follows if $G^3_4(c_i; R_{i_1},\dots, R_{i_6}) = 0$, for any $i = 1, \dots 4$ and any $(i_1, \dots, i_6) \in \{1,2,3\}^6$, where we have fixed $\{R_1, R_2, R_3, R_1+R_2+R_3\}$ as the frame of reference on $\PP^2_{\CC}$.
We will now prove the statement for a particular choice of $i$. It can be easily repeated for other possible choices. Let us assume $i = 1$. We want to show that:
\begin{align}
    G^3_4(c_1; R_{i_1},\dots, R_{i_6}) &= [1 \, 4 \, R_{i_1}][2 \, 8 \, R_{i_2}][3 \, 12 \, R_{i_3}][5 \, 6 \, R_{i_4}][7 \, 9 \, R_{i_5}][10 \, 11 \, R_{i_6}] + \label{generator G34}\\
    \hfill & \hspace{10pt} + [1 \, 7 \, R_{i_1}][2 \, 11 \, R_{i_2}][3 \, 6 \, R_{i_3}][4 \, 5 \, R_{i_4}][8 \, 9 \, R_{i_5}][10 \, 12 \, R_{i_6}] + \nonumber \\
    \hfill & \hspace{10pt} + [1 \, 10 \, R_{i_1}][2 \, 5 \, R_{i_2}][3 \, 9 \, R_{i_3}][4 \, 6 \, R_{i_4}][7 \, 8 \, R_{i_5}][11 \, 12 \, R_{i_6}] + \nonumber\\
    \hfill & \hspace{10pt}- [1 \, 4 \, R_{i_1}][2\, 11 \, R_{i_2}][3 \, 9 \, R_{i_3}][5 \, 6 \, R_{i_4}][7 \, 8 \, R_{i_5}][10 \, 12 \, R_{i_6}] + \nonumber\\
    \hfill & \hspace{10pt} - [1 \, 7 \, R_{i_1}][2\, 5 \, R_{i_2}][3 \, 12 \, R_{i_3}][4 \, 6 \, R_{i_4}][8 \, 9 \, R_{i_5}][10 \, 11 \, R_{i_6}] + \nonumber\\
    \hfill & \hspace{10pt} - [1 \, 10 \, R_{i_1}][2 \, 8 \, R_{i_2}][3 \, 6 \, R_{i_3}][4 \, 5 \, R_{i_4}][7 \, 9 \, R_{i_5}][11 \, 12 \, R_{i_6}] \nonumber \\
    \hfill &=0. \nonumber
\end{align} 
We will rewrite all terms of the polynomial as multiples of the first one. So,
we consider the following:
\begin{itemize}
    \item $\{1,4, 1+4\}$ is a basis for the projective line $r_1$. Thus, there exists a unique choice of $a, a', b, b'\in \CC$ such that: ($i$) $7 = a0+b4$ and ($ii$) $10 = a'0+b'4$.
    \item $\{2,8, 2+8\}$ is a basis for the projective line $r_2$. Thus, there exists a unique choice of $c, c', d, d'\in \CC$ such that: ($iii$) $5 = c2 + d8$ and ($iv$) $11 = c'2 + d'8$.
    \item $\{3,12, 3+12\}$ is a basis for the projective line $r_3$. Thus, there exists a unique choice of $e, e', f, f'\in \CC$ such that: ($v$) $6 = e3+f12$ and ($vi$) $9 = e'3+f'12$.
    \item $\{5,6, 5+6\}$ is a basis for the projective line $c_2$. Thus, there exists a unique choice of $\alpha, \beta \in \CC$ such that: ($vii$) $4 = \alpha 5 +\beta 6$.
    \item $\{7,9, 7+9\}$ is a basis for the projective line $c_3$. Thus, there exists a unique choice of $\gamma, \delta \in \CC$ such that: ($viii$) $8 = \gamma 7 +\delta 9$.
    \item $\{10,11, 10+11\}$ is a basis for the projective line $c_4$. Thus, there exists a unique choice of $\varepsilon, \zeta \in \CC$ such that: ($ix$) $12 = \varepsilon 10 +\zeta 11$.
\end{itemize}  Here, it is important to remark that $0, \dots, 12$ are fixed representatives of the corresponding points in $\PP^2_{\CC}$, which makes the choice of the coefficients in ($i$) - ($ix$) unique. Now, we exploit ($i$) - ($ix$) to modify the columns of the matrices showing up in the second, third, fourth, fifth and sixth summands of $G^3_4(c_1; R_{i_1},\dots, R_{i_6})$. By multi-linearity, we obtain:
\begin{align*}
    G^3_4(c_1; R_{i_1},\dots, R_{i_6}) &= \textcolor{red}{1} \cdot [1 \, 4 \, R_{i_1}][2 \, 8 \, R_{i_2}][3 \, 12 \, R_{i_3}][5 \, 6 \, R_{i_4}][7 \, 9 \, R_{i_5}][10 \, 11 \, R_{i_6}] + \\
    \hfill & \hspace{10pt} - \textcolor{orange}{b d'f\beta\gamma\zeta} [1 \, 4 \, R_{i_1}][2 \, 8 \, R_{i_2}][3 \, 12 \, R_{i_3}][5 \, 6 \, R_{i_4}][7 \, 9 \, R_{i_5}][10 \, 11 \, R_{i_6}] + \\
    \hfill & \hspace{10pt} - \textcolor{yellow}{b'df'\alpha\delta\varepsilon} [1 \, 4 \, R_{i_1}][2 \, 8 \, R_{i_2}][3 \, 12 \, R_{i_3}][5 \, 6 \, R_{i_4}][7 \, 9 \, R_{i_5}][10 \, 11 \, R_{i_6}] + \\
    \hfill & \hspace{10pt}- \textcolor{green}{d'f'\delta\zeta} [1 \, 4 \, R_{i_1}][2 \, 8 \, R_{i_2}][3 \, 12 \, R_{i_3}][5 \, 6 \, R_{i_4}][7 \, 9 \, R_{i_5}][10 \, 11 \, R_{i_6}]+\\
    \hfill & \hspace{10pt} - \textcolor{blue}{bd\alpha\gamma} [1 \, 4 \, R_{i_1}][2 \, 8 \, R_{i_2}][3 \, 12 \, R_{i_3}][5 \, 6 \, R_{i_4}][7 \, 9 \, R_{i_5}][10 \, 11 \, R_{i_6}] +\\
    \hfill & \hspace{10pt} - 
    \textcolor{purple}{b'f\beta\varepsilon} [1 \, 4 \, R_{i_1}][2 \, 8 \, R_{i_2}][3 \, 12 \, R_{i_3}][5 \, 6 \, R_{i_4}][7 \, 9 \, R_{i_5}][10 \, 11 \, R_{i_6}]\\
    \hfill &= (\textcolor{red}{1} -  \textcolor{orange}{b d'f\beta\gamma\zeta} - \textcolor{yellow}{b'df'\alpha\delta\varepsilon} - \textcolor{green}{d'f'\delta\zeta} - \textcolor{blue}{bd\alpha\gamma} - \textcolor{purple}{b'f\beta\varepsilon}) \cdot \\\hfill & \hspace{120pt} \cdot [1 \, 4 \, R_{i_1}][2 \, 8 \, R_{i_2}][3 \, 12 \, R_{i_3}][5 \, 6 \, R_{i_4}][7 \, 9 \, R_{i_5}][10 \, 11 \, R_{i_6}]
\end{align*}
As a consequence, the claim is equivalent to  $\textcolor{orange}{b d'f\beta\gamma\zeta} + \textcolor{yellow}{b'df'\alpha\delta\varepsilon} + \textcolor{green}{d'f'\delta\zeta} + \textcolor{blue}{bd\alpha\gamma} + \textcolor{purple}{b'f\beta\varepsilon} = \textcolor{red}{1}$. 
However,
\begin{align*}
4 &\stackrel{\footnotesize{(vii)}}{=} \underline{\alpha 5} + \underline{\beta 6}\\
\hfill &\stackrel{\footnotesize{(iii,v)}}{=} c\alpha 2  + e \beta 3 + \underline{d\alpha 8} + \underline{f \beta 12}\\
\hfill &\stackrel{\footnotesize{(viii,ix)}}{=} c\alpha 2  + e \beta 3 + \underline{d\alpha\gamma 7} + \underline{d\alpha\delta 9} + \underline{e f \beta 10} +  \underline{ f \beta \zeta 11} \\
\hfill &\stackrel{\footnotesize{(i,ii,iv,vi)}}{=} \ast_1 1 + \ast_2 2  + \ast_3 3 + (\textcolor{blue}{bd\alpha\gamma} + \textcolor{purple}{b'f\beta\varepsilon}) 4  + \underline{d'f\beta\zeta 8}+ \underline{df'\alpha\delta 12}\\
\hfill &\stackrel{\footnotesize{(viii,ix)}}{=} \ast_1 1 + \ast_2 2  + \ast_3 3 + (\textcolor{blue}{bd\alpha\gamma} + \textcolor{purple}{b'f\beta\varepsilon}) 4 + \underline{d'f\beta\gamma\zeta 7} + \underline{d'f\beta\delta\zeta 9} + \underline{df'\alpha\delta\varepsilon 10} +  \underline{df'\alpha\delta\zeta 11} \\
\hfill &\stackrel{\footnotesize{(i,ii,iv,vi)}}{=} \ast_1 1 + \ast_2 2  + \ast_3 3 + (\textcolor{orange}{b d'f\beta\gamma\zeta} + \textcolor{yellow}{b'df'\alpha\delta\varepsilon} + \textcolor{blue}{bd\alpha\gamma} + \textcolor{purple}{b'f\beta\varepsilon} )  4 + \underline{dd'f'\alpha\delta\zeta 8} + \underline{d'ff'\beta\gamma\zeta 12} \\
\hfill &\stackrel{\footnotesize{(iii,v)}}{=} \ast_1 1 + \ast_2 2  + \ast_3 3 + (\textcolor{orange}{b d'f\beta\gamma\zeta} + \textcolor{yellow}{b'df'\alpha\delta\varepsilon} + \textcolor{blue}{bd\alpha\gamma} + \textcolor{purple}{b'f\beta\varepsilon} )  4 + \underline{d'f'\alpha\delta\zeta 5 + d'f'\beta\delta\zeta 6}\\
\hfill &\stackrel{\footnotesize{(viii)}}{=} \ast_1 1 + \ast_2 2  + \ast_3 3 + (\textcolor{orange}{b d'f\beta\gamma\zeta} + \textcolor{yellow}{b'df'\alpha\delta\varepsilon} +  \textcolor{green}{d'f'\delta\zeta} + \textcolor{blue}{bd\alpha\gamma} + \textcolor{purple}{b'f\beta\varepsilon} )  4,
\end{align*}
which implies that:
\[(\textcolor{red}{1} -  \textcolor{orange}{b d'f\beta\gamma\zeta} - \textcolor{yellow}{b'df'\alpha\delta\varepsilon} - \textcolor{green}{d'f'\delta\zeta} - \textcolor{blue}{bd\alpha\gamma} - \textcolor{purple}{b'f\beta\varepsilon})4  = \ast_1 1 + \ast_2 2  + \ast_3 3.\]
Here, on the l.h.s., there is another representation of point $6$; conversely, on the r.h.s., there is a point in the line $c_1$. Since by construction $4 \notin c_1$, the equality above may hold if and only if both sides give $\left(\begin{smallmatrix} 0 \\ 0 \\ 0 \end{smallmatrix}\right)$, that is if and only if $\textcolor{orange}{b d'f\beta\gamma\zeta} + \textcolor{yellow}{b'df'\alpha\delta\varepsilon} + \textcolor{green}{d'f'\delta\zeta} + \textcolor{blue}{bd\alpha\gamma} + \textcolor{purple}{b'f\beta\varepsilon} = \textcolor{red}{1}$. The thesis follows.
\end{proof}

Furthermore, projections keep track of the vanishing of these polynomials.

\begin{lemma}\label{projective invariant property 2}
The vanishing of the polynomial $G^3_4 (c_r; R_{i_1}, \ldots,R_{i_6})$~is~a~projective~invariant~property. 
\end{lemma}

\begin{proof}
    We prove the lemma for the column $c_1$. We need to show that
    \[ G^3_4(c_1;P^1,\dots,P^6) = \sum _{\sigma \in \Sigma}\sgn (\sigma)[1 P_{\sigma}^1P^1][2 P_{\sigma}^2P^2][3 P_{\sigma}^3P^3][Q_{k_1}^{\sigma,1} Q_{k_1}^{\sigma,2}P^4][Q_{k_2}^{\sigma,1} Q_{k_2}^{\sigma,2}P^5][Q_{k_3}^{\sigma,1} Q_{k_3}^{\sigma,2}P^6] =0 \]
is a projective invariant property. We consider $T \in GL(\CC,3)$ and $D \in \diag(\CC,3)$ and we write down:
\begin{align*}
G^3_4(Tc_1D;TP^1D,\dots,TP^6D) &= \sum _{\sigma \in \Sigma}\sgn (\sigma)[T1D \;  TP_{\sigma}^1D \; TP^1D][T2D \; TP_{\sigma}^2D\; TP^2D] \cdot \\ 
\hfill & \hspace{57 pt} \cdot [T3D\;  TP_{\sigma}^3D \; TP^3D][TQ_{k_1}^{\sigma,1}D \; TQ_{k_1}^{\sigma,2}D \; TP^4D] \cdot \\ 
\hfill & \hspace{57 pt} \cdot[TQ_{k_2}^{\sigma,1}D \;  TQ_{k_2}^{\sigma,2}D \; TP^5D][TQ_{k_3}^{\sigma,1}D \; TQ_{k_3}^{\sigma,2}D \; TP^6D]\\
\hfill &= \det T^3 \det D^3 \bigg (\sum _{\sigma \in \Sigma}\sgn (\sigma)[1 P_{\sigma}^1P^1][2 P_{\sigma}^2P^2][3 P_{\sigma}^3P^3] \cdot \\
\hfill & \hspace{75pt} \cdot [Q_{k_1}^{\sigma,1} Q_{k_1}^{\sigma,2}P^4][Q_{k_2}^{\sigma,1} Q_{k_2}^{\sigma,2}P^5][Q_{k_3}^{\sigma,1} Q_{k_3}^{\sigma,2}P^6] \bigg ).
\end{align*}
\vspace{-1mm}
We observe that: 
\[G^3_4(Tc_1D;TP^1D,\dots,TP^6D) =0 \quad \Longleftrightarrow  \quad G^3_4(c_1;P^1,\dots,P^6) =0 ,\] 
which completes the proof for $c_1$. The analogous proof works for other choices of $c_i$. 
\end{proof}

The family of $G^3_4$ polynomials introduced in Theorem~\ref{CG34} characterises the liftability of 12-tuples of collinear points to a  3 $\times$ 4 grid. Specifically, the following holds.

\begin{theorem}\label{TFAE 3x4 grid}
    Let $r$ be a line in $\PP^2_{\CC}$ and $1, \dots, 12$  distinct points of $r$. Consider the collection $\LL=\{c_1 = \{1, 2, 3\}, c_2 = \{4, 5, 6\}, c_3 = \{7 , 8 , 9\}, c_4 = \{10, 11, 12\}, r_1 = \{1, 4, 7, 10\}, r_2 = \{2, 5, 8, 11\}, r_3 = \{3, 6 , 9 , 12\}\}$, where each of the tuples is the combinatorial line consisting of the points thereby contained. Then, the following statements are equivalent:
    \begin{itemize}
        \item[$i$.] The points $1, \dots, 12$ are projective image of a $3 \times 4$ grid.
        \item[$ii$.] The polynomials $G^3_4(c_i; P^1, \dots, P^6) = 0$ vanish for any $c_i \in \LL$ and any $P^1, \dots ,P^6 \in \PP^2_{\CC}$.
        \item[$iii$.] The polynomials $G^3_4(c_i; R_i, \dots, R_n)$ vanish for any $c_i \in \LL$ and any $(i,j,k,l,m,n) \in \{1,2,3\}^6$.
        \item[$iv$.] The polynomials $G^3_4(c_1; R_i, \dots , R_n)$ vanish for any $(i,j,k,l,m,n) \in \{1,2,3\}^6$.
        \item[$v$.] The polynomials $G^3_4(c_1;  R_i, \dots , R_n)$ vanish for any $(i,j,k,l,m,n) \in \{1,2,3\}^6$ with $i \leq j \leq k \leq l \leq m \leq n$.
    \end{itemize}
\end{theorem}

\begin{proof}
    It is immediate to see that $(ii) \Longrightarrow (iii) \Longrightarrow (iv) \Longrightarrow (v)$.\\

    \textbf{$(i) \Longrightarrow (ii)$} By Lemma~\ref{projective invariant property 2}, the vanishing of the $G^3_4$ polynomials is a projective invariant property, and so, if twelve collinear points are the projective image of a $3 \times 4$ grid, then $G^3_4(c_i; P^1, \dots, P^6) = 0$ where $c_i, P^1, \dots, P^6$ satisfy the assumptions of Theorem~\ref{CG34}.

    \smallskip

    \textbf{$(ii) \Longrightarrow (i)$} We now consider twelve collinear points $1, \dots, 12$ which satisfy the vanishing of the polynomials $G^3_4(c_i; P^1, \dots, P^6)$, for any choice of $c_i \in \LL$, and any choice of $P^1, \dots, P^6 \in \PP^2_{\CC}$. Let $P$ be a point which does not lie on the line of points $1, \dots, 12$.
    
    \begin{claim}
        The condition that the polynomials $G^3_4(c_i; P^1, \dots, P^6)$ vanish for any choice of $c_i \in \LL$ and any choice of $P^1, \dots, P^6 \in \PP^2_{\CC}$ implies that the collinearity matrix $\Lambda_{G^3_4}$ has rank $\leq 9$.
        
        $$ \Lambda_{G^3_4} = \left( \begin{smallmatrix}
            [2\,3] & -[1\,3] & [1\,2] & 0 & 0 & 0 & 0 & 0 & 0 & 0 & 0 & 0 \\
            0 & 0 & 0 & [5\,6] & -[4\,6] & [4\,5] & 0 & 0 & 0 & 0 & 0 & 0 \\
            0 & 0 & 0 & 0 & 0 & 0 & [8\,9] & -[7\,9] & [7\,8] & 0 & 0 & 0 \\
            0 & 0 & 0 & 0 & 0 & 0 & 0 & 0 & 0 & [11\,12] & -[10\,12] & [10\,11]\\
            [4\,7] & 0 & 0 & -[1\,7] & 0 & 0 & [1\,4] & 0 & 0 & 0 & 0 & 0 \\
            [4\,10] & 0 & 0 & -[1\,10] & 0 & 0 & 0 & 0 & 0 &[1\,4] & 0 & 0 \\
            [7\,10] & 0 & 0 & 0 & 0 & 0 & -[1 \, 10] & 0 & 0 & [1\,7] & 0 & 0 \\
            0 & 0 & 0 & [7 \, 10] & 0 & 0 & -[4 \, 10] & 0 & 0 & [4\,7] & 0 & 0\\
            0 & [5\,8] & 0 & 0 & -[2\,8] & 0 & 0 & [2\,5] & 0 & 0 & 0 & 0  \\
            0 & [5\,11] & 0 & 0 & -[2\,11] & 0 & 0 & 0 & 0 & 0 & [2\,5] & 0  \\
            0 & [8\,11] & 0 & 0 & 0 & 0 & 0 & -[2 \, 11] & 0 & 0 & [2\,8] & 0 \\
            0 & 0 & 0 & 0 & [8\, 11] & 0 & 0 & -[5 \, 11] & 0 & 0 & [5\,8] & 0 \\
            0 & 0 & [6\,9] & 0 & 0 & -[3\,9] & 0 & 0 & [3\,6] & 0 & 0 & 0  \\
            0 & 0 & [6\,12] & 0 & 0 & -[3\,12] & 0 & 0 & 0 & 0 & 0 & [3\,6] \\
            0 & 0 & [9\,12] & 0 & 0 & 0 & 0 & 0 & -[3 \, 12] & 0 & 0 & [3\,9] \\
            0& 0 & 0 & 0 & 0 & [9\, 12] & 0 & 0 & -[6 \, 12] & 0 & 0 & [6\,9] 
        \end{smallmatrix}\right)$$
    \end{claim}

\begin{proof}[\bf Proof of the claim.]
We aim to prove that all the 10-minors of the matrix $\Lambda_{G^3_4}$ vanish when evaluated at the coordinates of points $1$ through $12$. In considering the 10-minors of the matrix $\Lambda_{G^3_4}$, we note that each line contains the variables of three collinear points. Specifically, we can introduce the following partition of the rows of the matrix $\Lambda_{G^3_4}$:
    \begin{itemize}
        \item Rows $\mathscr{C} = \{I, II, III, IV\}$ correspond to the columns of the grid.
        \item Rows $\mathscr{R}_1 = \{V, VI, VII, VIII\}$ correspond to the first row of the grid.
        \item Rows $\mathscr{R}_2 = \{IX, X, XI, XII\}$ correspond to the second row of the grid.
        \item Rows $\mathscr{R}_3 = \{XIII, XIV, XV, XVI\}$ correspond to the third row of the grid.
    \end{itemize}        
        The submatrices of $\Lambda_{G^3_4}$ formed by the rows in one of the $\mathscr{R}i$'s and the corresponding non-zero columns have rank 2. Thus, whenever a $10 \times 10$ submatrix of $\Lambda{G^3_4}$ contains three rows of one of the $\mathscr{R}_i$'s, the corresponding minor is automatically 0. Consequently, the only minors that can possibly be non-zero are those generated by a selection of rows including the four rows from $\mathscr{C}$ and two rows from each $\mathscr{R}_i$. Furthermore, the choice of two rows in each of the $\mathscr{R}_i$'s does not affect the final computation of a 10-minor. This is because the determinant of a matrix remains unchanged if one replaces two rows with linear combinations of them.
        
        To sum up, after having chosen a 10 $\times$ 10 submatrix of $\Lambda_{G^3_4}$ there are only three possible patterns for its non-zero entries, up to switching its columns and rows.
        $$ \left( \begin{smallmatrix}
            \circ & & & & \circ & & & & \circ & \\
            & \circ & & & & \circ & & & & \circ \\
            & & \circ & & & & \circ & & & \\
            & & & \circ & & & & \circ & & \\
            \circ & \circ & \circ & \circ & & & & & & \\
             \circ & \circ & \circ & \circ & & & & & & \\
            & & & & \circ & \circ & \circ & \circ & & \\
            & & & & \circ & \circ & \circ & \circ & & \\
            & & & & & & & & \circ & \circ \\
            & & & & & & & & \circ & \circ
        \end{smallmatrix}\right) \qquad 
         \left( \begin{smallmatrix}
            \circ & & & & \circ & & & \circ & &\\
            & \circ & & & & \circ & & & \circ &\\
            & & \circ & & & & \circ & & & \circ\\
            & & & \circ & & & & & & \\
            \circ & \circ & \circ & \circ & & & & & & \\
             \circ & \circ & \circ & \circ & & & & & & \\
            & & & & \circ & \circ & \circ & & & \\
            & & & & \circ& \circ & \circ & & & \\
            & & & & & & & \circ & \circ & \circ \\
            & & & & & & & \circ & \circ & \circ
        \end{smallmatrix}\right) \qquad
         \left( \begin{smallmatrix}
            \circ & & & & \circ & & & \circ & &\\
            & \circ & & & & \circ & & & \circ &\\
            & & \circ & & & & \circ & & &\\
            & & & \circ & & & & & & \circ\\
            \circ & \circ & \circ & \circ & & & & & & \\
             \circ & \circ & \circ & \circ & & & & & & \\
            & & & & \circ & \circ & \circ & & & \\
            & & & & \circ& \circ & \circ & & & \\
            & & & & & & & \circ & \circ & \circ \\
            & & & & & & & \circ & \circ & \circ
        \end{smallmatrix}\right) 
        $$
In each of the matrices among the patterns above, there are three clear blocks of columns. Each of these blocks corresponds to one of the rows of the grid. Since the rows of the grid play a symmetric role, up to relabelling the points, we can associate the leftmost block with $r_1$, the central block with $r_2$, and the rightmost block with $r_3$. This reduces the study cases to: 6 minors having the first pattern, 4 minors having the second pattern, and 12 minors having the third pattern.

        These 10-minors are all contained in the ideal generated by the polynomials $G^3_4(c_i; (0,0,1)^t)$, as shown in the computations available on \href{https://github.com/ollieclarke8787/PointAndLineConfigurations}{GitHub}.
    \end{proof}
    It follows from the claim that the linear system $\Lambda_{G^3_4} (z_1 \dots z_{12})^t = (0 \dots 0)^t$ has a solution space of dimension at least 3. In other words, it is possible to choose $z_1, \dots, z_{12}$ such that
    $$\begin{cases}
    \scriptstyle{[2\,3] \cdot z_1 -[1\,3] \cdot z_2 + [1\,2] \cdot z_3 = [5\,6] \cdot z_4 -[4\,6] \cdot z_5 + [4\,5] \cdot z_6 = [8\,9] \cdot z_7 -[7\,9] \cdot z_8 + [7\,8] \cdot z_9 = [11\,12] \cdot z_{10} - [10\,12] \cdot z_{11} + [10\,11] \cdot z_{12} = 0}\\
    
    \scriptstyle{[4\,7] \cdot z_1 - [1\,7] \cdot z_4 + [1\,4] \cdot z_7 = [4\,10] \cdot z_1 -[1\,10] \cdot z_4 + [1\,4] \cdot z_{10} = [7\,10] \cdot z_1  -[1 \, 10] \cdot z_7 + [1\,7] \cdot z_{10} = [7 \, 10] \cdot z_4 -[4 \, 10] \cdot z_7 + [4\,7]\cdot z_{10} = 0} \\

    \scriptstyle{[5\,8] \cdot z_2 -[2\,8] \cdot z_5 + [2\,5] \cdot z_8 = [5\,11] \cdot z_2 -[2\,11] \cdot z_5 + [2\,5] \cdot z_{11} = [8\,11] \cdot z_2 -[2 \, 11] \cdot z_8 + [2\,8] \cdot z_{11} = [8\, 11] \cdot z_5 -[5 \, 11] \cdot z_8 + [5\,8] \cdot z_{11} = 0 }\\
    
    \scriptstyle{[6\,9] \cdot z_3 - [3\,9] \cdot z_6  + [3\,6] \cdot z_9 = [6\,12] \cdot z_3 -[3\,12] \cdot z_6 + [3\,6] \cdot z_{12} = [9\,12] \cdot z_3 -[3 \, 12] \cdot z_9 + [3\,9] \cdot z_{12} = [9\, 12] \cdot z_6 -[6 \, 12] \cdot z_9 + [6\,9] \cdot z_{12} = 0}
    \end{cases}$$
        and such that the points $\left(\begin{smallmatrix} x_i \\ 1 \\ z_i \end{smallmatrix} \right)$ for $i = 1, \dots, 12$ span the whole projective plane. This ensures the existence of a non-degenerate $3 \times 4$ grid whose image via the projection through $P$ on the line $r$ consists exactly of points $1, \dots, 12$.

        \smallskip

    \textbf{$(iii) \Longrightarrow (ii)$} The implication follows because the polynomials in ($ii$) and ($iii$) generate the same ideal, by the multilinearity of determinants.

       \smallskip

    \textbf{$(v) \Longrightarrow (iii)$} The implication follows from direct computation. For each polynomial $g = G^3_4(c_i; R_i, \dots, R_n)$, we verify that $g$ belongs to the ideal generated by $3$-minors arising from the collinearity constraints, and the polynomials in $(v)$. The code is available on \href{https://github.com/ollieclarke8787/PointAndLineConfigurations}{GitHub}.
    \end{proof}

    This characterisation is crucial to provide a minimal generating set for the ideal $I_{G^3_4}$.
\begin{theorem}\label{generators grid}
    Let $G^3_4$ be the simple matroid underlying the grid configuration with $3$ rows and $4$ columns, and let $I_{G^3_4}$ denote the ideal of the matroid variety. Then,
    \begin{multline*}I_{G^3_4} = \langle [1 \, 2 \, 3], \dots, [10 \, 11 \, 12], G^3_4(c_{1}; R_i, R_j, R_k, R_l, R_m,R_n) \, \forall \, i \leq j\leq k \leq l \leq m \leq n, \\ \hbox{with } i,j,k,l,m,n \in \{1,2,3\} \rangle,
    \end{multline*}
    where $[1 \, 2 \, 3], \dots, [10\, 11\, 12]$ are the collinearities given by the grid.
\end{theorem}

\begin{proof}
For the ease of notation, we denote $I$ for the ideal in the right hand side of the above equation.
    By Theorem~\ref{CG34}, we have that $I \subset I_{G^3_4}$. We prove the converse inclusion in the following 3 steps.

\medskip
\noindent{$\pmb{V(I)=V(I_{G^3_4})}$.} 
We prove this equality by double inclusion. Since $I \subset I_{G^3_4}$, we have that $V_{G^3_4} \subseteq V(I)$. We prove  $V(I) \subset V(G^3_4)$. Let $A \in \CC^{36}$ be a point in $V(I)$. The coordinates $A^1_1, A^2_1, A^3_1, \dots, A^1_{12}, A^2_{12}, A^3_{12}$ of $A$ can be seen as the $(x,y,z)$-coordinates of 12 points in the projective plane. These can be represented with the following $3 \times 12$ matrix:
    \[ A = \left( \begin{matrix}
        A_1^1 & A_2^1& \dots & A^1_{12} \\
        A_1^2 & A_2^2 & \dots & A^2_{12} \\
        A_1^3 & A_2^3 & \dots & A^3_{12}
    \end{matrix} \right).\]
    The columns of $A$ generate a realisable matroid $M_A$ which corresponds to a point-line configuration $\CCC_A$ (as the columns of $A$ are coordinates of points).
Since $A \in V(I)$, $A$ satisfies the determinantal collinearity conditions. Hence, $A$ is a point in the combinatorial closure of the matroid associated to the grid-configuration $G^3_4$. Corollary~\ref{decomposition quasi-liftable} ensures that $A$ is close either to a realisation of the matroid corresponding to a line with 12 marked points or to a realisation of $G^3_4$. 

    The fact that $A \in V(I)$ implies that the coordinates of $A$ satisfy also the $G^3_4$ polynomials, implying that, if the points of $A$ lie on a line, then they can be lifted to a $3 \times 4$ grid.
    This is true also in the case of having 11 points on a line and a loop, indeed we can always find a 12th point on the line such that $1, \dots, 12$ are the projective image of a $3 \times 4$ grid. Among the 11 points there is a $3 \times 3$ sub-grid which can always be lifted by Remark~\ref{lifting 3x3}. The two other points are determined by the intersection of this grid with the fibres and the last one comes consequently (see Figure~\ref{fig:3x4 grid lifting}). Now, via the perturbation procedure, we conclude that $A$ is in the Euclidean closure of $\Gamma_{G^3_4}$. Hence, $A \in V(I_{G^3_4})$.

\begin{figure}[h]
        \centering
        \includegraphics[scale=0.3]{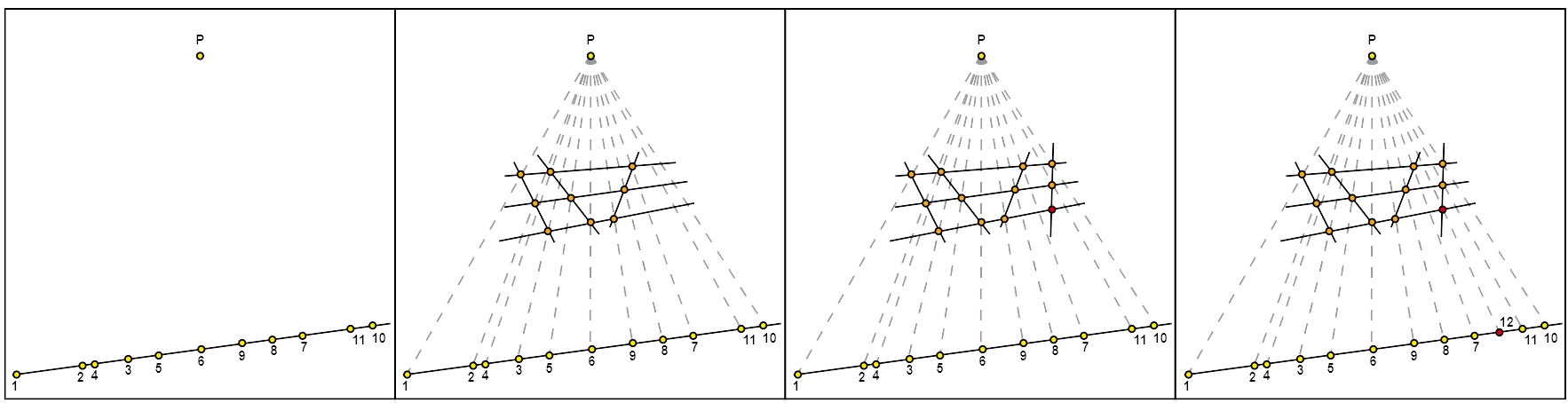}
        \caption{From left to right, this figure justify the existence of at least a choice of 12 completing $1, \dots, 11$ to a projective image of a $3\times4$ grid.}
        \label{fig:3x4 grid lifting}
    \end{figure}

\medskip
\noindent{\bf The ideal $\pmb{I_{G^3_4}}$ is radical.} 
The circuit ideal of the $3\times 4$ grid is radical by the explicit computation in \cite{PfSt}, where the authors show that the circuit ideal is the intersection of two prime ideals, hence it is radical. We also have that $I_{G^3_4} = \sqrt{I}$. In addition, in the proof of Theorem~\ref{TFAE 3x4 grid}, we have computed in \textit{Macaulay2} that $I$ is equal to one of the prime ideals in their decomposition, by reducing their generators modulo the $G^3_4$ polynomials, which all reduce to zero. Thus, as $I$ is a radical ideal, $I_{G^3_4} = I$.

    \medskip
\noindent{\bf Minimal generating set.} To prove the minimality of the generating set, we show that none of the generators belong to the ideal generated by the others. This is verified numerically, and the code is available on \href{https://github.com/ollieclarke8787/PointAndLineConfigurations}{GitHub}.
\end{proof}

As a direct 
consequence of Corollary~\ref{decomposition quasi-liftable} and Theorem~\ref{generators grid} we have:
\begin{corollary}
    The circuit variety $V_{\C (G^3_4)}$ decomposes irreducibly as:
\[      
V_{\C (G^3_4)}= V_0 \cup V_{QS}= V(\langle [i j k] \, | \, \{i,j,k\} \subseteq \{1, \dots, 12\} \rangle) \cup V(I_{G^3_4}).
\]
\end{corollary}

\medskip
\noindent{\bf Acknowledgement.}
G.M. and F.M. were supported by the FWO grant G023721N. F.M. was partially supported by the grants G0F5921N (Odysseus programme) from the Research Foundation - Flanders (FWO), iBOF/23/064 from KU Leuven and the UiT Aurora project MASCOT. G.M. would like to thank Emiliano Liwski for helpful discussions about the proofs of Lemma~\ref{non-simple} and Theorem~\ref{circuit variety of liftable conf.}.

\printbibliography

\medskip
\noindent 
\small{\textbf{Authors' addresses}

\noindent
School of Mathematics, University of Edinburgh, Edinburgh, United Kingdom
\\ E-mail address: {\tt oliver.clarke@ed.ac.uk}

\medskip  \noindent
Department of Computer Science, KU Leuven, Celestijnenlaan 200A, B-3001 Leuven, Belgium\\ 
   Department of Mathematics, KU Leuven, Celestijnenlaan 200B, B-3001 Leuven, Belgium\\ E-mail address: {\tt fatemeh.mohammadi@kuleuven.be}

\medskip  \noindent
Department of Mathematics, KU Leuven, Celestijnenlaan 200B, B-3001 Leuven, Belgium
\\ E-mail address: {\tt giacomo.masiero@kuleuven.be}
}
\end{document}